\documentclass[11pt,a4paper]{amsart}

\usepackage{amssymb, amsthm, amsmath, amsfonts, mathrsfs, datetime, multirow}
\usepackage{fancyhdr}
\usepackage{indentfirst}
\usepackage{cmap}
\usepackage{caption}
\usepackage{tikz-cd} 
\usetikzlibrary{calc, 3d, patterns}
\tikzset{
	labl/.style={anchor=south, rotate=-90, inner sep=.6mm}
}
\usepackage{dsfont}
\usepackage{stmaryrd}
\usepackage{enumitem}

\usepackage{color}\definecolor{darkblue}{rgb}{0,0.1,.5}
\usepackage[colorlinks=true,linkcolor=darkblue, urlcolor=darkblue, citecolor=darkblue]{hyperref} 
\usepackage[left=1.5cm, right=1.5cm, top=3cm, bottom=3cm]{geometry}
\definecolor{newab}{rgb}{.9,0,0}

\calclayout

\makeatletter
\renewcommand{\@makecaption}[2]{%
	\vspace{\abovecaptionskip}%
	\sbox{\@tempboxa}{#1. #2}
	\ifdim \wd\@tempboxa >\hsize
	#1. #2\par
	\else
	\global\@minipagefalse
	\hbox to \hsize {\hfil #1. #2\hfil}%
	\fi
	\vspace{\belowcaptionskip}}
\makeatother 

\def\R{\mathbb{R}}
\def\C{\mathbb{C}}

\def\Z{\mathbb{Z}}

\def\sk{\mathcal K}
\def\sK{\mathcal K}

\def\vphi{\varphi}
\def\k{\Bbbk}
\def\m{\mathfrak m}
\def\n{\mathfrak n}

\def\Tor{\mathop{\mathrm{Tor}}\nolimits}
\def\Cotor{\mathop{\mathrm{Cotor}}\nolimits}
\def\id{\mathds{1}}
\DeclareMathOperator{\lcm}{lcm}
\DeclareMathOperator{\sign}{sign}

\def\zk{\mathcal {Z_K}}	
\def\cpk{(\C P^{\infty})^{\sk}}
\def\pt{\mathit{pt}}

\def\geq{\geqslant}
\def\le{\leqslant}

\newcommand{\hr}[2][]{\hyperref[#2]{#1~\ref{#2}}}

\newtheorem{theorem}{Theorem}[section]
\newtheorem*{theorem*}{Theorem}
\newtheorem{lemma}[theorem]{Lemma}
\newtheorem{problem}[theorem]{Problem}
\newtheorem{proposition}[theorem]{Proposition}

\theoremstyle{definition}
\newtheorem{construction}[theorem]{Construction}
\newtheorem{definition}[theorem]{Definition}
\newtheorem{example}[theorem]{Example}
\newtheorem{remark}[theorem]{Remark} 

\renewenvironment{proof}[1][{\itshape Proof}]{{\itshape #1. }}{\qed}

\numberwithin{equation}{section}

\title[Whitehead products and substitution of simplicial complexes]{Higher Whitehead products in moment-angle complexes and substitution of simplicial complexes}

\author{Semyon Abramyan}
\address{AG Laboratory, HSE, 6 Usacheva str., Moscow, Russia, 119048}
\email{semyon.abramyan@gmail.com}

\author{Taras Panov}
\address{Department of Mathematics and Mechanics, Lomonosov Moscow
State University, Leninskie gory, 119991 Moscow, Russia;\newline
Institute for Information Transmission Problems, Russian Academy of Sciences, Moscow;\newline
Institute of Theoretical and Experimental Physics, Moscow}
\email{tpanov@mech.math.msu.su}

\thanks{The first author was partially supported by the Russian Academic Excellence Project `5-100', by the Russian Foundation for Basic Research (grant no.~ 18-51-50005), and by the Simons Foundation.}
\thanks{The second author was partially supported by the Russian Foundation for Basic Research (grants no.~17-01-00671, 18-51-50005), and by the Simons Foundation.}

\begin{document}

\begin{abstract}
We study the question of realisability of iterated higher Whitehead products with a given form of nested brackets by simplicial complexes, using the notion of the moment-angle complex~$\zk$. Namely, we say that a simplicial complex $\sk$ realises an iterated higher Whitehead product $w$ if $w$ is a nontrivial element of $\pi_*(\zk)$.  The combinatorial approach to the question of realisability uses the operation of substitution of simplicial complexes: for any iterated higher Whitehead product~$w$ we describe a simplicial complex $\partial\Delta_w$ that realises~$w$. Furthermore, for a particular form of brackets inside~$w$, we prove that $\partial\Delta_w$ is the smallest complex that realises~$w$. We also give a combinatorial criterion for the nontriviality of the product~$w$. In the proof of nontriviality we use the Hurewicz image of~$w$ in the cellular chains of~$\zk$ and the description of the cohomology product of~$\zk$.
The second approach is algebraic: we use the coalgebraic versions of the Koszul and Taylor complex for the face coalgebra of~$\sK$ to describe the canonical cycles corresponding to iterated higher Whitehead products~$w$. This gives  another criterion for realisability of $w$.
\end{abstract}

\dedicatory{Dedicated to our Teacher Victor Matveevich Buchstaber on the occasion of his 75th birthday}

\maketitle

\tableofcontents


\section{Introduction}
Higher Whitehead products are important invariants of unstable homotopy type. They have been studied since the 1960s in the works of homotopy theorists such as Hardie~\cite{hard61}, Porter~\cite{port65} and Williams~\cite{will72}.

The appearance of moment-angle complexes and, more generally, polyhedral products in toric topology at the end of the 1990s brought a completely new perspective on higher homotopy invariants such as higher Whitehead products. The homotopy fibration of polyhedral products
\begin{equation}\label{zkfib}
  (D^2,S^1)^\sK\to(\mathbb C P^\infty)^\sK\to (\mathbb C P^\infty)^m
\end{equation}
was used as the universal model for studying iterated higher Whitehead products in~\cite{pa-ra08}. Here $(D^2,S^1)^\sK=\zk$ is the moment-angle complex, and $(\mathbb C P^\infty)^\sK$ is homotopy equivalent to the Davis--Januszkiewicz space~\cite{bu-pa00,bu-pa15}. The form of nested  brackets in an iterated higher Whitehead product is reflected in the combinatorics of the simplicial complex~$\sk$.

There are two classes of simplicial complexes $\sk$ for which the moment-angle complex is particularly nice. From the geometric point of view, it is interesting to consider complexes $\sk$ for which $\zk$ is a manifold. This happens, for example, when $\sk$ is a simplicial subdivision of sphere or the boundary of a polytope. The resulting moment-angle manifolds~$\zk$ often have remarkable geometric properties~\cite{pano13}. On the other hand, from the homotopy-theoretic point of view, it is important to identify the class of simplicial complexes $\sk$ for which the moment-angle complex $\zk$ is homotopy equivalent to a wedge of spheres. We denote this class by~$B_\Delta$.  The spheres in the wedge are usually expressed in terms of iterated higher Whitehead products of the canonical $2$-spheres in the polyhedral product~$(\mathbb C P^\infty)^\sK$. We denote by $W_\Delta$ the subclass in $B_\Delta$  consisting of those $\sk$ for which $\zk$ is  a wedge of iterated higher Whitehead products. The question of describing the class $W_\Delta$ was studied in~\cite{pa-ra08} and formulated explicitly in~\cite[Problem~8.4.5]{bu-pa15}. It follows from the results of~\cite{pa-ra08} and~\cite{g-p-t-w16} that $W_\Delta=B_\Delta$ if we restrict attention to \emph{flag} simplicial complexes only, and a flag complex $\sK$ belongs to~$W_\Delta$ if and only if its one-skeleton is a chordal graph. Furthermore, it is known that $W_\Delta$ contains directed $MF$-complexes~\cite{gr-th16}, shifted and totally fillable complexes~\cite{ir-ki1,ir-ki2}. On the other hand, it has been recently shown in~\cite{abra} that the class $W_\Delta$ is \emph{strictly} contained in~$B_\Delta$. There is also a related question of \emph{realisability} of an iterated higher Whitehead product $w$ with a given form of nested brackets: we say that a simplicial complex $\sk$ \textit{realises} an iterated higher Whitehead product $w$ if $w$ is a nontrivial element of $\pi_*(\zk)$ (see \hr[Definition]{wdelta}). For example, the boundary of simplex $\sK=\partial\Delta(1,\ldots,m)$ realises a single (non-iterated) higher Whitehead product $[\mu_1,\ldots,\mu_m]$, which maps $\zk=S^{2m-1}$ into the fat wedge $(\C P^\infty)^\sK$.

We suggest two approaches to the questions above. The first approach is combinatorial: using the operation of substitution of simplicial complexes (\hr[Section]{subst}), for any iterated higher Whitehead product~$w$ we describe a simplicial complex $\partial\Delta_w$ that realises~$w$ (\hr[Theorem]{realisation}). Furthermore, for a particular form of brackets inside~$w$, we prove in \hr[Theorem]{smallest_realisation}~(a) that $\partial\Delta_w$ is the smallest complex that realises~$w$. We also give a combinatorial criterion for the nontriviality of the product~$w$ (\hr[Theorem]{smallest_realisation}~(b)). In the proof of nontriviality we use the Hurewicz image of~$w$ in the cellular chains of~$\zk$ and the description of the cohomology product of~$\zk$ from~\cite{bu-pa00}. \hr[Theorems]{realisation},\hr{smallest_realisation} and further examples  not included in this paper lead us to conjecture that $\partial\Delta_w$ is the smallest complex realising~$w$,  for any iterated higher Whitehead product (see \hr[Problem]{proble}).

The second approach is algebraic: we use the coalgebraic versions of the Koszul complex and the Taylor resolution of the face coalgebra of~$\sK$ to describe the canonical cycles corresponding to iterated higher Whitehead products~$w$. This gives  another criterion for realisability of $w$ in \hr[Theorem]{tarep}.

\section{Preliminaries}
\label{prelimi}

A \textit{simplicial complex} $\sk$ on the set $[m] = \{1, 2, \ldots, m\}$ is a collection of subsets $I \subset [m]$ closed under taking any subsets. We refer to $I\in\sk$ as a \textit{simplex} or a \textit{face} of~$\sk$, and always assume that $\sk$ contains $\varnothing$ and all singletons $\{i\}$, $i=1,\ldots,m$. We do not distinguish between $\sk$ and its geometric realisation when referring to the homotopy or topological type of~$\sk$.  

We denote by $\Delta^{m-1}$ or $\Delta(1, \ldots, m)$ the full simplex on the set $[m]$. Similarly, denote by $\Delta(I)$ a simplex with the vertex set $I\subset[m]$ and denote its boundary  by $\partial\Delta(I)$. A \emph{missing face}, or a \emph{minimal non-face} of~$\sk$ is a subset $I\subset[m]$ such that $I\notin\sk$, but $\partial\Delta (I)\subset \sk$.

Assume we are given a set of $m$ pairs of based cell complexes
\[
	(\underline X, \underline A) = \{(X_1, A_1),\dots, (X_m, A_m)\}
\]
where $A_i\subset X_i$. For each simplex $I\in\sk$ we set
\[
	(\underline X, \underline A)^I = \{(x_1, \ldots, x_m)\in X_1\times\cdots\times X_m\;|\; x_j \in A_j \text{ for $j\notin I$}\}.
\]
The \textit{polyhedral product} of $(\underline X, \underline A)$ corresponding to $\sk$ is the following subset of $X_1\times\dots\times X_m$:
\[
	(\underline X, \underline A)^{\sk} = 
  \bigcup\limits_{I\in \sk} 
  (\underline X, \underline A)^I\qquad (\subset X_1\times\cdots\times X_m).
\]

In the case when $(X_i, A_i)=(D^2, S^1)$ for each $i$, we use the notation $\zk$ for $(D^2,S^1)^{\sk}$, and refer to $\zk = (D^2, S^1)^{\sk}$ as the \textit{moment-angle complex}. Also, if $(X_i, A_i)=(X,\pt)$ for each $i$, where $\pt$ denotes the basepoint, we use the abbreviated notation $X^{\sk}$ for $(X,\pt)^{\sk}$.

\begin{theorem}[{\cite[Theorem~4.3.2]{bu-pa15}}]\label{cofib}
	The moment-angle complex $\zk$ is the homotopy fibre of the canonical inclusion $\cpk \hookrightarrow (\C P^{\infty})^m$.
\end{theorem}

There is also the following more explicit description of the fibre inclusion $\mathcal Z_{\sk} \to \cpk$ in~\eqref{zkfib}. Consider the map of pairs $(D^2, S^1) \to (\C P^{\infty}, pt)$ sending the interior of the disc homeomorphically onto the complement of the basepoint in $\C P^1$. By the functoriality, we have the induced map of the polyhedral products $\mathcal{Z_K} = (D^2, S^1)^{\sk} \to \cpk$.

The general definition of higher Whitehead products can be found
in~\cite{hard61}. We only describe Whitehead products in the space
$\cpk$ and their lifts to $\zk$. In this case
the indeterminacy of higher Whitehead products can be controlled
effectively because extension maps can be chosen canonically.

Consider the \emph{$i$th coordinate map}
\[
  \mu_i\colon(D^2, S^1) \to S^2 \cong \C P^1 \hookrightarrow (\C P^{\infty})^{\vee m}  
  {}\hookrightarrow \cpk.
\]
Here the second map is the canonical inclusion of $\C P^1$ into the $i$-th summand of the wedge. The third map is induced by the embedding of $m$ disjoint points into $\sk$.
The \emph{Whitehead product} (or \emph{Whitehead bracket}) $[\mu_i, \mu_j]$ of $\mu_i$ and $\mu_j$
is the homotopy class of the map
\[
  S^3 \cong \partial D^4 \cong \partial (D^2\times D^2) \cong (D^2\times S^1)\cup
  (S^1\times D^2) \xrightarrow{[\mu_i,\mu_j]} \cpk
\]
where 
\[
[\mu_i, \mu_j](x, y) =
\begin{cases}
\mu_i(x) \quad \text{for $(x,y)\in D^2\times S^1$};\\
\mu_j(y) \quad \text{for $(x,y)\in S^1\times D^2$}.
\end{cases}   
\]

Every Whitehead product $[\mu_i,\mu_j]$ becomes trivial after composing with the embedding
$\cpk \hookrightarrow (\C P^{\infty})^m \simeq K(\Z^m, 2)$.
This implies that
$[\mu_i, \mu_j]\colon S^3 \to (\C P^{\infty})^{\sk}$ lifts to the fibre $\zk$, as shown next:
\begin{center}
\begin{tikzcd}
\zk\arrow{r} & \cpk \arrow{r} & (\C P^{\infty})^m\\
&S^3\arrow[u,"{[\mu_i,\mu_j]}"']\arrow[dashed]{ul}
\end{tikzcd}
\end{center}
We use the same notation $[\mu_i,\mu_j]$ for a lifted map $S^3\to\zk$. Such a lift can be chosen canonically as the inclusion of a subcomplex
\[
  [\mu_i, \mu_j]\colon S^3 \cong(D^2\times S^1)\cup(S^1\times D^2) \hookrightarrow \zk.
\]

The Whitehead product $[\mu_i, \mu_j]$
is trivial if and only if the map $[\mu_i, \mu_j]\colon S^3 \to \zk$
can be extended to a map
$D^4 \cong D^2_i\times D^2_j \hookrightarrow\zk$. This is equivalent to the condition that $\Delta(i,j)=\{i,j\}$ is a $1$-simplex of~$\sk$.

\emph{Higher Whitehead products} are defined inductively as follows.
Let $\mu_{i_1},\dots,\mu_{i_n}$ be a collection of maps such
that the $(n-1)$-fold product
\[
  [\mu_{i_1},\dots,\widehat{\mu_{i_k}},\dots, \mu_{i_n}]\colon S^{2(n-1)-1}\to(\C P^{\infty})^{\sk}
\]
is trivial for any $k$.
Then there exists a \emph{canonical}
extension $\overline{[\mu_{i_1},\dots,\widehat{\mu_{i_k}},\dots, \mu_{i_n}]}$
to a map from $D^{2(n-1)}$ given by the composite
\[
\overline{[\mu_{i_1},\ldots, \widehat{\mu_{i_k}},\ldots, \mu_{i_n}]}\colon
D^2_{i_1}\times\cdots\times D^2_{i_{k-1}}\times D^2_{i_{k+1}}\times\cdots\times
D^2_{i_n}\hookrightarrow \zk \to (\C P^{\infty})^{\sk}.
\] 
Furthermore, all these extensions are compatible on the subproducts corresponding to the vanishing brackets of shorter length.
The \emph{$n$-fold product} $[\mu_{i_1},\ldots, \mu_{i_n}]$ is defined
as the homotopy class of the map
\[
S^{2n-1} \cong \partial(D^2_{i_1}\times\cdots\times D^2_{i_n})\cong
\bigcup\limits_{k=1}^n(D^2_{i_1}\times\cdots\times S^1_{i_k}\times\cdots
\times D^2_{i_n})\xrightarrow{[\mu_{i_1},\ldots, \mu_{i_n}]}
\cpk
\]
which is given by
\[
	[\mu_{i_1},\ldots, \mu_{i_n}](x_1,\ldots, x_n) =
	\overline{[\mu_{i_1},\ldots,\widehat\mu_{i_k},\ldots, \mu_{i_n}]}
		(x_1,\ldots,\widehat x_k,\ldots, x_n)\quad\text{if }\;x_k\in S^1_{i_k}.
\]

In~\hr[Proposition]{singlent} below we show that $[\mu_{i_1},\ldots,\mu_{i_p}]$ is defined in $\pi_{2p-1}(\cpk)$ if and only if $\partial \Delta(i_1,\ldots,i_p)$ is a subcomplex of~$\sk$, and $[\mu_{i_1},\ldots,\mu_{i_p}]$ is trivial if and only if $\Delta(i_1,\ldots,i_p)$ is a simplex of~$\sk$.  

Alongside with higher Whitehead products of canonical coordinate maps $\mu_i$
we consider \emph{general iterated} higher 
Whitehead products, i.\,e. higher Whitehead products in which
arguments can be higher Whitehead products.
For example, 
\[
  \Bigl[\mu_1,\mu_2,[\mu_3,\mu_4,\mu_5],
  \bigl[\mu_6,\mu_{13},[\mu_7,\mu_8,\mu_9],\mu_{10}\bigr],[\mu_{11},\mu_{12}]\Bigr].
\]
Among general iterated higher Whitehead products we distinguish \emph{nested} products, which have the form
\[
  w = \Big[\big[\dots\big[[\mu_{i_{11}},\ldots, \mu_{i_{1p_1}}], \mu_{i_{21}},\dots, \mu_{i_{2p_2}}\big],  
  \dots\big], \mu_{i_{n1}},\dots, \mu_{i_{np_n}}\Big] \colon S^{d(w)}\to\cpk.
\]
Here $d(w)$ denotes the dimension of $w$. Sometimes we refer to $[\mu_{i_1}, \dots, \mu_{i_p}]$ as a \emph{single} (noniterated) higher Whitehead product. 

As in the case of ordinary Whitehead products any iterated higher Whitehead product lifts to a map $S^{d(w)} \to\zk$ for dimensional reasons.    

\begin{definition}\label{wdelta} We say that a simplicial complex $\sk$ \textit{realises} a higher iterated Whitehead product $w$ if $w$ is a nontrivial element of $\pi_*(\zk)$.
\end{definition}

\begin{example}\label{singlereal}
The complex $\partial \Delta(i_1,\ldots,i_p)$ realises the single higher Whitehead product $[\mu_{i_1},\ldots,\mu_{i_p}]$. 
\end{example}

\begin{construction}[cell decomposition of~$\zk$]\label{celld}
Following~\cite[\S4.4]{bu-pa15}, we decompose the disc $D^2$ into 3 cells: the point $1\in D^2$ is the 0-cell; the complement to $1$ in the boundary circle is the 1-cell, which we denote by~$S$; and the interior of $D^2$ is the 2-cell, which we denote by~$D$. These cells are canonically oriented as subsets of $\R^2$. By taking products we obtain a cellular decomposition of~$(D^2)^m$, in which cells are encoded by pairs of subsets $J,I\subset [m]$ with $J\cap I=\varnothing$: the set $J$ encodes the $S$-cells in the product and $I$ encodes the $D$-cells. We denote the cell of $(D^2)^m$ corresponding to a pair $J,I$ by $\varkappa(J,I)$:
\begin{align*}
  \varkappa(J, I)& =\prod_{i\in I}D_i\times\prod_{j\in J}S_j\\
   &  = \big\{(x_1, \ldots, x_m) \in (D^2)^m\;\big|\\
   &\qquad\qquad\qquad\text{$x_i \in D$ for $i\in I$, $x_j \in S$ for $j   \in J$ and $x_l = 1$ for $l\notin J\cup I$} \big\}.
\end{align*}
Then $\zk$ is a cellular subcomplex in~$(D^2)^m$; we have $\varkappa(J,I)\subset\zk$ whenever $I\in\sk$.

Given a subset $J\subset [m]$, we denote by $\sk_J$ the \emph{full subcomplex} of~$\sk$ on~$J$, that is, 
\[
  \sk_J = \{I\in \sk|\, I\subset J\}.
\]
Let $C_{p-1}(\sk_J)$ denote the group of $(p-1)$-dimensional simplicial chains of~$\sk_J$; its basis consists of simplices $L\in\sk_J$, $|L|=p$. We also denote by $\mathcal C_q(\zk)$ the group of $q$-dimensional cellular chains of~$\zk$ with respect to the cell decomposition described above.
\end{construction}

\begin{theorem}[see~{\cite[Theorems~4.5.7, 4.5.8]{bu-pa15}}]\label{hochster}
	The homomorphisms
	\[
		C_{p-1}(\sk_J) \longrightarrow \mathcal C_{p+|J|}(\mathcal{Z_K}),\quad
     {L\mapsto \sign(L,J)\varkappa(J\backslash L, L)}
	\]
	induce injective homomorphisms
	\begin{align*}
		\widetilde H_{p-1}(\sk_J) \hookrightarrow H_{p+|J|}(\mathcal{Z_K}),
	\end{align*}
which are functorial with respect to simplicial inclusions. Here $L\in \sk_J$ is a simplex, and $\sign(L,J)$ is the sign of the shuffle $(L,J)$. The inclusions above induce an isomorphism of abelian groups
	\[
		\bigoplus\limits_{J\subset [m]}\widetilde H_*(\sk_J) 
     \stackrel\cong\longrightarrow H_*(\mathcal{Z_K}).
	\]
The cohomology versions of these isomorphisms combine to form a ring isomorphism
\[
	\bigoplus\limits_{J\subset [m]}\widetilde H^*(\sk_J) 
   \stackrel\cong\longrightarrow H^*(\mathcal{Z_K}).
\]
where the ring structure on the left hand side is given by the
maps
\[
  H^{k-|I|-1}(\sk_{I})\otimes H^{\ell-|J|-1}(\sk_{J})\to
  H^{k+\ell-|I|-|J|-1}(\sk_{I\cup J})
\]
which are induced by the canonical simplicial inclusions $\sk_{I\cup
J}\to\sk_I\mathbin{*}\sk_J$ for $I\cap J=\varnothing$ and are zero
for $I\cap J\ne\varnothing$.
\end{theorem}

\section{The Hurewicz image of a higher Whitehead product}\label{hirewicz_image_section}

Here we consider the Hurewicz homomorphism $h\colon\pi_*(\zk)\to H_*(\zk)$. The canonical
cellular chain representing the Hurewicz image $h(w)\in H_*(\zk)$ of a \emph{nested} higher Whitehead product~$w$ was described in~\cite{abra}.

\begin{lemma}[{\cite[Lemma~4.1]{abra}}]\label{hurewicz_image_old}
The Hurewicz image
\[
h\left(\Big[\big[\dots\big[[\mu_{i_{11}},\ldots, \mu_{i_{1p_1}}], \mu_{i_{21}},\dots, \mu_{i_{2p_2}}\big],  
\dots\big], \mu_{i_{n1}},\dots, \mu_{i_{np_n}}\Big]\right) \in H_{2(p_1 + \cdots + p_n) - n}(\mathcal {Z_K})
\]
is represented by the cellular chain
\[
  h_c(w)=\prod\limits_{k = 1}^n \left(\sum\limits_{j = 1}^{p_k} D_{i_{k1}}\cdots    
  D_{i_{k(j-1)}}S_{i_{kj}}D_{i_{k(j+1)}}\cdots D_{i_{kp_k}} \right).
\]
\end{lemma}

A more general version of this lemma is presented next. It gives a simple recursive formula describing the canonical cellular chain $h_c(w)$ which represents the Hurewicz image of a \emph{general} iterated higher Whitehead product $w \in \pi_*(\zk)$, therefore providing an effective method of identifying nontrivial Whitehead products in the homotopy groups of a moment-angle complex~$\zk$. Some applications are also given below.

\begin{lemma}\label{hurewicz_image}
Let $w$ be a general iterated higher Whitehead product
$$
w = [w_1, \dots, w_q, \mu_{i_1}, \dots, \mu_{i_p}] \in \pi_*(\zk).
$$
Here $w_k$ is a \textup(general iterated\textup) higher Whitehead product for $k = 1, \dots, q$. Then the Hurewicz image $h(w)\in H_*(\zk)$ is represented by the following canonical cellular chain:
\[
  h_c(w) = h_c(w_1)\cdots h_c(w_q)\Big(\sum\limits_{k=1}^p D_{i_1}\cdots D_{i_{k-1}}S_{i_k} 
  D_{i_{k+1}}\cdots D_{i_p}\Big).
\]
\end{lemma}

We shall refer to $h_c(w)$ as the \emph{canonical} cellular chain for an interated higher Whitehead product~$w$.
In the case of nested products, \hyperref[hurewicz_image]{Lemma~\ref{hurewicz_image}} reduces to \hyperref[hurewicz_image_old]{Lemma~\ref{hurewicz_image_old}}.

\begin{proof}[Proof of \textup{\hr[Lemma]{hurewicz_image}}]
Let $d, d_1, \dots, d_q$ be the dimensions of $w, w_1, \dots, w_q$, respectively. The Whitehead product $w$ is represented by the composite map
	{\small
		\begin{multline}\label{31}
		S^d \cong \partial(D^{d_1}\times\dots\times D^{d_q}\times D^2_{i_1}\times\cdots\times 
		D^2_{i_p})\\
		\cong 
		\bigg(D^{d_1}\times\dots\times D^{d_q}\times
		\Big(\bigcup\limits_{k=1}^p D^2_{i_1}\times\dots\times S^1_{i_k}\times\dots\times
		D^2_{i_p}\Big)\bigg)
		\\\cup
		\bigg(
		\Big(\bigcup_{l=1}^q  D^{d_1}\times\dots\times S^{d_l-1}\times\dots\times D^{d_q}\Big)
		\times D^2_{i_1}\times\dots\times D^2_{i_p}\bigg)
		\\ \stackrel\gamma\longrightarrow
		\bigg(S^{d_1}\times\dots\times S^{d_q}\times
		\Big(\bigcup\limits_{k=1}^p D^2_{i_1}\times\dots\times S^1_{i_k}\times\dots\times   
		D^2_{i_p}\Big)\bigg)
		\\\cup
		\bigg(\Big(\bigcup_{l=1}^q S^{d_1}\times\dots\times\pt\times\dots\times S^{d_q}\Big)
		\times D^2_{i_1}\times\dots\times D^2_{i_p}\bigg)
		\to \zk.
		\end{multline}
	}%
The map $\gamma$ above contracts the boundary of each $D^{d_l}$, $l = 1, \ldots, q$. Note that the whole cartesian product in the last row above has dimension less than $d$, so its Hurewicz image is trivial.

Using the same argument for the spheres $S^{d_1}, \dots, S^{d_q}$, we obtain that $w$ factors through a map from $S^d$ to a union of products of discs and circles, which embeds as a subcomplex in~$\zk$. By the induction hypothesis each sphere $S^{d_k}, k = 1, \dots, q$, maps to the subcomplex of~$\zk$ corresponding to the cellular chain $h_c(w_k)$. Therefore, by~\eqref{31}, the Hurewicz image of $w$ is represented by the subcomplex corresponding to the product of $h_c(w_1), \ldots, h_c(w_q)$ and $\sum\limits_{k=1}^p D_{i_1}\cdots D_{i_{k-1}} S_{i_k} D_{i_{k+1}}\cdots D_{i_p}$.
\end{proof}

As a first corollary we obtain a combinatorial criterion for the nontriviality of a single higher Whitehead product.

\begin{proposition}\label{singlent}
A single higher Whitehead product $[\mu_{i_1},\ldots,\mu_{i_p}]$ is
\begin{itemize}[leftmargin=0.06\textwidth]
\item[\textup(a\textup)] defined in $\pi_{2p-1}(\cpk)$ \textup(and lifts to $\pi_{2p-1}(\zk)$\textup) if and only if $\partial \Delta(i_1,\ldots,i_p)$ is a subcomplex of~$\sk$;
\item[\textup(b\textup)] trivial if and only if $\Delta(i_1,\ldots,i_p)$ is a simplex of~$\sk$.  
\end{itemize}
\end{proposition}
\begin{proof}
If the Whitehead product $[\mu_{i_1},\ldots,\mu_{i_p}]$ is defined, then each $(p-1)$-fold product $[\mu_{i_1},\dots, \widehat\mu_{i_k}\dots,\mu_{i_p}]$ is trivial. By the induction hypothesis, this implies that $\partial \Delta(i_1,\ldots,i_p)$ is a subcomplex of~$\sk$. 

Suppose that $\Delta(i_1,\ldots,i_p)$ is not a simplex of $\sk$. Then, by \hr[Lemma]{hurewicz_image}, the Hurewicz image $h([\mu_{i_1},\ldots,\mu_{i_p}])$ gives a nontrivial homology class in $H_*(\zk)$ corresponding to $[\partial \Delta(i_1,\ldots,i_p)]\in\widetilde H_*(\sk_{i_1,\dots, i_p})$ via the isomorphism of \hr[Theorem]{hochster}. Thus, $[\mu_{i_1},\ldots,\mu_{i_p}]$ is itself nontrivial.
\end{proof}

This proposition will be generalised to iterated higher Whitehead products in \hr[Section]{realisation_section}.

\medskip

\hr[Lemmata]{hurewicz_image_old}, \ref{hurewicz_image} and \hyperref[hochster]{Theorem~\ref{hochster}} can be used to detect simplicial complexes $\sk$ for which $\zk$ is a wedge of iterated higher Whitehead products. We recall the following definition.

\begin{definition}
A simplicial complex $\sk$ belongs to the class $W_{\Delta}$ if $\zk$ is a wedge of spheres, and each sphere in the wedge is a lift of a linear combination of iterated higher Whitehead products.
\end{definition}

As a first example of application of our method we deduce the results of Iriye and Kishimoto that shifted and totally fillable complexes belong to the class $W_{\Delta}$.

\begin{example}\label{example_ik1}
A simplicial complex $\sk$ is called \emph{shifted} if its vertices can be ordered in such way that the following condition is satisfied: whenever $I\in \sk$, $i\in I$ and $j>i$, we have $(I-i)\cup j\in \sk$.

Let $\mathop\mathrm{MF}\nolimits_m(\sk)$ be the set of missing faces of $\sk$ containing the maximal vertex $m$, i.\,e.
\[
  \mathop\mathrm{MF}\nolimits_m(\sk) = 
  \{I\subset [m]\,|\, I\notin \sk,\; \partial\Delta (I)\subset \sk \text{ and } m\in I\}.
\]
As observed in~\cite{ir-ki1}, for a shifted complex $\sk$ there is a homotopy equivalence
\begin{equation}\label{shifted}
  \sk \simeq \bigvee\limits_{I \in \mathop\mathrm{MF}\nolimits_m(\sk)}\partial \Delta(I)
\end{equation}
(the reason is that the quotient $\sk/\mathop{\mathrm{star}}_m\sk$ is homeomorphic to the wedge on the right hand side of~\eqref{shifted}, by definition of a shifted complex).
Note that a full subcomplex of a shifted complex is again shifted. Then \hyperref[hochster]{Theorem~\ref{hochster}} together with
\eqref{shifted} implies that $H_*(\zk)$ is a free abelian group generated by the homology classes of cellular chains of the form
\begin{equation}\label{ik1_chain}
  \Big(\sum\limits_{l=1}^p D_{i_1}\cdots D_{i_{l-1}} S_{i_l} D_{i_{l+1}}\cdots D_{i_p}\Big)S_{j_1}
  \cdots S_{j_q}
\end{equation}
where $I = \{i_1, \dots, i_p\}\in \mathop\mathrm{MF}\nolimits_m(\sk_{i_1,\ldots,i_p,j_1,\dots, j_q})$. \hyperref[hurewicz_image]{Lemma~\ref{hurewicz_image_old}} implies that~\eqref{ik1_chain} is the canonical cellular chain for the nested Whitehead product
\[
  w = \Big[\Big[\Big[\dots\big[[\mu_{i_1}, \dots, \mu_{i_p}], \mu_{j_1}\big],\dots \Big], 
  \mu_{j_{q-1}}\Big], \mu_{j_q}\Big].
\]
Hence, the following wedge of the Whitehead products
\[
  \bigvee\limits_{J\subset[m]}\bigvee\limits_{\substack{I = \{i_1, \dots, i_p\} \in 
  \mathop\mathrm{MF}_m(\sk_J)\\ J\setminus I=\{j_1, \dots, j_q\}}} 
  \Big[\Big[\Big[\dots\big[[\mu_{i_1}, \dots, \mu_{i_p}], \mu_{j_1}\big],\dots \Big], \mu_{j_{q-1}}\Big],   
  \mu_{j_q}\Big]\colon 
  \!\!\!\!\!\bigvee\limits_{\substack{J\subset [m]\\ 
  I\in\mathop\mathrm{MF}_m(\sk_J)}} S^{d(w)}_{J, I}\to \zk
\]
induces an isomorphism in homology, so it is a homotopy equivalence. Thus, we obtain the following.

\begin{theorem}[{\cite{ir-ki1}}]
Every shifted complex $\sk$ belongs to $W_{\Delta}$.
\end{theorem}
\end{example}

Here is another result which can be proved using \hyperref[hurewicz_image]{Lemma~\ref{hurewicz_image}}.

\begin{example}\label{example_ik2}
A simplicial complex $\sk$ is called \emph{fillable} if there is a collection $\mathop\mathrm{MF}_{\rm fill}(\sk)$ of missing faces $I_1, \dots, I_k$ such that $\sk\cup I_1 \cup\dots\cup I_k$ is contractible. If any full subcomplex of $\sk$ is fillable, then $\sk$ is called \emph{totally fillable}.

Note that homology of any full subcomplex $\sk_J$ in a totally fillable complex $\sk$ is generated by the cycles $\partial\Delta(I)$ for $I\in \mathop\mathrm{MF}_{\mathrm{fill}}(\sk_J)$. As in \hyperref[example_ik1]{Example~\ref{example_ik1}}, $H_*(\zk)$ is a free abelian group generated by the homology classes of cellular chains
\[
  \Big(\sum\limits_{l=1}^p D_{i_1}\dots D_{i_{l-1}} S_{i_l} D_{i_{l+1}}\dots D_{i_p}\Big)
  S_{j_1}\dots S_{j_q},
\]
where $\Delta(i_1, \dots, i_q) \in \mathop\mathrm{MF}_{\rm fill}(\sk_{j_1, \dots, j_p, i_1, \dots, i_q})$. Again, the map
\[
  \bigvee\limits_{J\subset[m]}\bigvee\limits_{\substack{I
  \in \mathop\mathrm{MF}_{\rm fill}(\sk_J)\\ J\setminus I=\{j_1, \dots, j_q\}}}   
  \Big[\Big[\Big[\dots\big[[\mu_{i_1}, \dots, \mu_{i_p}], \mu_{j_1}\big],\dots \Big], \mu_{j_{q-1}}\Big],   
  \mu_{j_q}\Big]\colon \bigvee\limits_{\substack{J\subset [m]\\ 
  I\in\mathop\mathrm{MF}_{\mathrm{fill}}(\sk_J)}} S^{d(w)}_{J, I}\to \zk
\]
is a homotopy equivalence, by the same reasons. We obtain the following.
\begin{theorem}[{\cite{ir-ki2}}]
Every totally fillable complex $\sk$ belongs to $W_{\Delta}$.
\end{theorem}
\end{example}

\section{Substitution of simplicial complexes}\label{subst}
The combinatorial construction presented here is similar to the~one described in \cite{ayze13} and \cite{b-b-c-g}, although the resulting complexes are different. An analogous construction for building sets was suggested by N.~Erokhovets (see~\cite[Construction~1.5.19]{bu-pa15}).

\begin{definition}\label{substitution_definition}
Let $\sk$ be a simplicial complex on the set $[m]$, and let $\sk_1, \dots, \sk_m$ be a set of $m$ simplicial complexes. We refer to the simplicial complex
\begin{equation}\label{substitution_equation}
\sk(\sk_1, \dots, \sk_m) = \{I_{j_1}\sqcup\dots\sqcup I_{j_k}\,|\; I_{j_l}\in \sk_{j_l},\; l = 1,\dots, k \quad\text{and}\quad \{j_1, \dots, j_k\}\in \sk\}
\end{equation}
as the \emph{substitution} of $\sk_1, \dots, \sk_m$ into $\sk$.

The set of missing faces $\mathop\mathrm{MF}(\sk(\sk_1, \dots, \sk_m))$ of a substitution complex can be described as follows. First, every missing face of each $\sk_i$ is the missing face of $\sk(\sk_1, \dots, \sk_m)$. Second, for every missing face $\Delta(i_1, \dots, i_k)$ of $\sk$ we have the following set of missing faces of the substitution complex:
\[
	\mathop\mathrm{MF}\nolimits_{i_1,\dots, i_k}\bigl(\sk(\sk_1, \dots, \sk_m)\bigr) = 
  \bigl\{\Delta(j_1, \dots, j_k)\,|\; j_l \in \sk_{i_l},\; l =1,\dots, k\bigr\}.
\]
It is easy to see that there are no other missing faces in $\sk(\sk_1, \dots, \sk_m)$, so we have
\[
	\mathop\mathrm{MF}\bigl(\sk(\sk_1, \dots, \sk_m)\bigr) = \mathop\mathrm{MF}  
  (\sk_1)\sqcup\dots\sqcup\mathop\mathrm{MF}
  (\sk_m)\sqcup\hspace{-8mm}\bigsqcup\limits_{\Delta(i_1, \dots, i_k)\in\mathop\mathrm{MF}
  (\sk)}\hspace{-8mm}\mathop\mathrm{MF}\nolimits_{i_1,\dots, i_k}
  \bigl(\sk(\sk_1, \dots, \sk_m)\bigr).
\]
\end{definition}

\begin{figure}[ht]
\begin{tikzpicture}
\coordinate (A1) at ( 0cm, 3cm);
\coordinate (A2) at ( 0cm,-3cm);
\coordinate (A3) at (-2cm, 0cm);
\coordinate (A4) at ( .75cm,-.75cm);
\coordinate (A5) at ( 2cm,  0cm);
	
\fill[gray!20, opacity = 0.3] (A1) -- (A4) -- (A3) -- cycle;
\fill[gray!20, opacity = 0.3] (A1) -- (A4) -- (A5) -- cycle;
\fill[gray!20, opacity = 0.7]  (A2) -- (A4) -- (A5) -- cycle;
\fill[gray!20, opacity = 0.5]  (A2) -- (A4) -- (A3) -- cycle;
	
\foreach \j in {3, 4, 5}	
{
\draw[thick] (A1) -- (A\j) ;
\draw[thick] (A2) -- (A\j) ;
}
\draw[thick] (A3) -- (A4) -- (A5);
\draw[thick, dashed] (A1) -- (A2);
\draw[thick, dashed] (A3) -- (A5);
\foreach \i in {1, 2, ..., 5}
{
\draw[fill=black] (A\i) circle (0.2em);
}
	
\draw[fill = black] (0cm, -0.3cm) circle (0.1em);
	
\draw (A1) node[above right] {\textbf 4};
\draw (A2) node[right] 	      {\textbf 5};
\draw (A3) node[above left] {\textbf 1};
\draw (A4) node[below right] {\textbf 2};
\draw (A5) node[above right] {\textbf 3};
\end{tikzpicture}
\caption{Substitution complex $\partial\Delta(\partial\Delta(1,2,3), 4, 5)$}
\label{12345}
\end{figure}
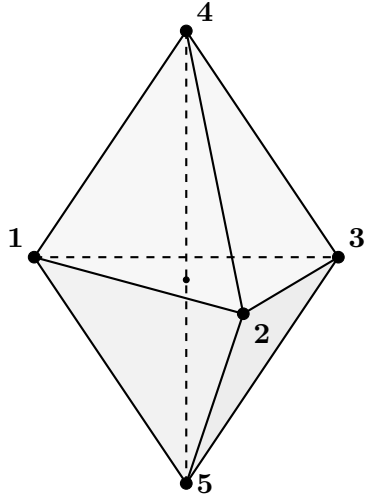

\begin{example}
If each $\sk_i$ is a point $\{i\}$, then $\sk(\sk_1, \dots, \sk_m) = \sk$. In particular, $\partial\Delta^{m-1}(1, \dots, m)= \partial\Delta^{m-1}$. In the case of substitution into a simplex $\Delta^{m-1}$ or its boundary $\partial\Delta^{m-1}$ we shall omit the dimension, so we have $\partial\Delta(1, \dots, m) = \partial\Delta^{m-1}$, which is compatible with the previous notation.
\end{example}

The next example is our starting point for further generalisations.

\begin{example}
Let $\sk = \partial\Delta^{m-1}$ and each $\sk_i$ is a point, except for $\sk_1$. We have $\partial\Delta(\sk_1, i_2,\dots, i_m) = \mathcal J_{m-2}(\sk_1)$, where $\mathcal J_{n}(\mathcal L)$ is the operation defined in \cite[Theorem~5.2]{abra}. 
By \cite[Theorem~6.1]{abra}, the iterated substitution
\[
\partial\Delta\big(\partial\Delta(j_1, \dots, j_q), i_1, \dots, i_p\big)
\]
is the smallest simplicial complex that realises the Whitehead product
\[
\big[[\mu_{j_{1}},\dots, \mu_{j_q}],\mu_{i_1},\dots, \mu_{i_p}\big].
\]

The case $q = 3$, $p = 2$ is shown in  \hr[Figure]{12345}.
\end{example}

The next example will be used in \hr[Theorem]{smallest_realisation}.

\begin{construction}\label{main_substitution}
Here we inductively describe the canonical simplicial complex $\partial\Delta_w$ associated with a general iterated higher Whitehead product~$w$.

We start with the boundary of simplex $\partial\Delta(i_1, \dots, i_m)$ corresponding to a single higher Whitehead product $[\mu_{i_1}, \dots, \mu_{i_m}]$. Now we write a general iterated higher Whitehead product recursively as
\[
w = [w_1, \dots, w_q, \mu_{i_1}, \dots, \mu_{i_p}] \in \pi_*(\zk),
\]
where $w_1,\ldots,w_q$ are nontrivial general iterated higher Whitehead products, $q\geq0$. We assign to $w$ 
the substitution complex 
\[
\partial\Delta_w \overset{\rm def}{=} \partial\Delta(\partial\Delta_{w_1}, \dots, \partial\Delta_{w_q}, i_1,\dots, i_p).
\]
We also define recursively the following subcomplex of $\partial\Delta_w$:
\[
\partial\Delta^{\rm sph}_w = \partial\Delta^{\rm sph}_{w_1}*\dots*\partial\Delta^{\rm sph}_{w_q}*\partial\Delta(i_1, \dots, i_p).
\]
By definition, $\partial\Delta_{w}^{\rm sph}$ is a join of boudaries of simplices, so it is homeomorphic to a sphere. Furthermore, $\dim\partial\Delta^{\rm sph}_w=\dim\partial\Delta_w$.

We refer to the subcomplex $\partial\Delta^{\rm sph}_w$ as the \emph{top sphere of $\partial\Delta_w$}.

For example, the top sphere of $\partial\Delta(\partial\Delta(1,2,3), 4, 5)$ is obtained by deleting the edge $\Delta(4, 5)$, see \hr[Figure]{12345}. 
\end{construction}

\begin{proposition}\label{wedge}
The complex $\partial\Delta_w$ is homotopy equivalent to a wedge of spheres, and the top sphere $\partial\Delta^{\rm sph}_w$ represents the sum of top-dimensional spheres in the wedge.
\end{proposition}
\begin{proof}
By construction, $\partial\Delta_w$ is obtained from a sphere $\partial\Delta^{\rm sph}_w$ by attaching simplices of dimension at most $\dim\partial\Delta^{\rm sph}_w$. It follows that the attaching maps are null-homotopic, which implies both statements.
\end{proof}

\section{Realisation of higher Whitehead products}\label{realisation_section}

Given an iterated higher Whitehead product~$w$, we show that the substitution complex $\partial\Delta_w$ realises~$w$. Furthermore, for a particular form of brackets inside~$w$, we prove that $\partial\Delta_w$ is the smallest complex that realises~$w$. We also give a combinatorial criterion for the nontriviality of the product~$w$.

Recall from \hr[Proposition]{singlent} that a single higher Whitehead product $[\mu_{i_1},\ldots,\mu_{i_p}]$ is realised by the complex $\partial \Delta(i_1,\ldots,i_p)$. 

\begin{theorem}\label{realisation}
Let $w_1,\ldots,w_q$ be nontrivial iterated higher Whitehead products. 
The complex $\partial\Delta_w$ described in \textup{\hr[Construction]{main_substitution}} realises the iterated higher Whitehead product 
\begin{equation}\label{product_realisation}
  w = [w_1, \dots, w_q, \mu_{i_1}, \dots, \mu_{i_p}].
\end{equation}
\end{theorem}
\begin{proof}
To see that product~\eqref{product_realisation} is defined in $\mathcal Z_{\partial\Delta_w}$ we need to construct the corresponding map $S^{d(w)}\to\mathcal Z_{\partial\Delta_w}$. This is done precisely as described in the proof  of \hr[Lemma]{hurewicz_image}. Furthermore, \hr[Lemma]{hurewicz_image}
gives the cellular chain $h_c(w)\in\mathcal C_*(\mathcal Z_{\partial\Delta_w})$ representing the Hurewicz image~$h(w)\in H_*(\mathcal Z_{\partial\Delta_w})$. The cellular chain $h_c(w)\in\mathcal C_*(\mathcal Z_{\partial\Delta_w})$ corresponds to the simplicial chain $\partial\Delta_w^{\rm sph}\in C_*(\partial\Delta_w)$ via the isomorphism of \hr[Theorem]{hochster}. Now \hr[Proposition]{wedge} implies that the simplicial  homology class $[\partial\Delta_w^{\rm sph}]\in H_*(\partial\Delta_w)$ is nonzero. Thus, $h(w)\ne0$ and the Whitehead product $w$ is nontrivial.
\end{proof}

\smallskip

For a particular configuration of nested brackets, a more precise statement holds.

\begin{theorem}\label{smallest_realisation}
Let $w_j = [\mu_{j_1}, \dots, \mu_{j_{p_j}}]$, $j = 1,\dots, q$,
be nontrivial single higher Whitehead products. 
Consider an iterated higher Whitehead product
\[
w = [w_1, \dots, w_q, \mu_{i_1}, \dots, \mu_{i_p}].
\]
Then the product $w$ is
\begin{itemize}[leftmargin=0.06\textwidth]
\item[\textup(a\textup)]\label{statement_a}
defined in $\pi_*(\zk)$ if and only if $\sk$ contains $\partial\Delta_w=\partial\Delta\big(\partial\Delta_{w_1}, \dots, \partial\Delta_{w_q}, i_1, \dots, i_p\big)$ as a subcomplex, where $\partial\Delta_{w_j} = \partial\Delta(j_1, \dots, j_{p_j}), j = 1,\dots, q$; 
\item[\textup(b\textup)]\label{statement_b}
trivial in $\pi_*(\zk)$ if and only if $\sk$ contains
$$\Delta\big(\partial\Delta_{w_1}, \dots, \partial\Delta_{w_q}, i_1, \dots, i_p\big) = \partial\Delta_{w_1}*\dots*\partial\Delta_{w_q}*\Delta(i_1,\dots, i_p)$$ as a subcomplex.
\end{itemize}
\end{theorem}

Note that assertion~(a) implies that $\partial\Delta_w$ is the smallest simplicial complex realising the Whitehead product~$w$.

\begin{proof}
We may assume that $q>0$; otherwise the theorem reduces to the \hr[Proposition]{singlent}.
We consider three cases: $p = 0$; $p = 1$; $p>1$.

\smallskip

\underline{The case $p =0$.}
We have $w = [w_1, \dots, w_q]$. 

We first prove assertion~(b).
Let $d_1, \dots, d_q$ and $d = d_1+\dots+d_q -1$ be the dimensions of the Whitehead products $w_1, \dots, w_q$ and $[w_1, \dots, w_q]$, respectively. The condition that~$w$ vanishes implies the existence of the dashed arrow in the diagram
\begin{center}
\begin{tikzcd}
S^d \arrow{r}\arrow[hookrightarrow]{d} & \mathop{\mathrm{FW}}(S^{d_1}, \dots, S^{d_q}) \arrow{r}\arrow[hookrightarrow]{d} & \zk \\
D^{d+1} \arrow{r} & S^{d_1}\times\dots\times S^{d_q} \arrow[dashed]{ur}
\end{tikzcd}
\end{center}
Here $\mathop{\mathrm{FW}}(S^{d_1}, \dots, S^{d_q})$ denotes the fat wedge of spheres $S^{d_1}, \dots, S^{d_q}$, and the top left arrow is the attaching map of the top cell.

Let $\sigma_j \in H^{d_j}(\zk)$ be the cohomology class dual to the sphere $S^{d_j} \subset \mathop{\mathrm{FW}}(S^{d_1}, \dots, S^{d_q})$, $j=1,\dots, q$. By the assumption, the single Whitehead product $w_j$ is nontrivial, which implies that $\sigma_j\neq 0$ (see \hr[Propostion]{singlent}). The class $\sigma_j \in H^{d_j}(\zk)$ corresponds to the simplicial cohomology class  $\big[\partial\Delta_{w_j}\big]^* \in \widetilde H^*(\sk_{\partial\Delta_{w_j}})$ via the cohomological version of the isomorphism of \hr[Theorem]{hochster}. Here $\sk_{\partial\Delta_{w_j}}$ is the full subcomplex $\partial\Delta_{w_j}$ of~$\sk$.
Since the Whitehead product $[w_1, \dots, w_q]$ is trivial, the cohomology product $\sigma_1\cdots\sigma_q$ is nontrivial in $H^*(\zk)$ (see the diagram above). By the cohomology product description in
\hr[Theorem]{hochster}, this implies that $\sk$ contains $\partial\Delta_1*\dots*\partial\Delta_{w_q}$ as a full subcomplex, and assertion~(b) follows.

To prove assertion~(a), note that the existence of the product $[w_1, \dots,w_q]$ implies that each product $[w_1, \dots, \widehat{w_j},\dots, w_q]$, $j = 1,\dots, q$, is trivial. By assertion~(b), complex $\sk$ contains the union $\bigcup\limits_{j = 1}^q \partial\Delta_{w_1}*\dots*\widehat{\partial\Delta_{w_j}}*\dots*\partial\Delta_{w_q}$ which is precisely $\partial\Delta(\partial\Delta_{w_1}, \dots, \partial\Delta_{w_q})$. This finishes the proof for the case $p = 0$.

\smallskip

\underline{The case $p=1$.}
We have $w = [w_1, \dots, w_q, \mu_{i_1}]$. 

We first prove (b), that is, assume $w=0$. This implies that $[w_1, \dots, w_q]=0$. By the previous case, we know that $\sk$ contains $\Delta(\partial\Delta_{w_1},\dots, \partial\Delta_{w_q})$ as a full subcomplex. 
We need to prove that $\sk$ contains $\Delta(\partial\Delta_{w_1},\dots, \partial\Delta_{w_q})*\Delta(i_1)$, which is a cone with apex~$i_1$. The Hurewicz image $h(w)\in H_*(\zk)$ is zero, because $w$ is trivial. Therefore, the canonical cellular chain 
$h_c(w)=h_c(w_1)\cdots h_c(w_q)S_{i_1}$ (see \hr[Lemma]{hurewicz_image}) is a boundary. By \hr[Theorem]{hochster}, this implies that the simplicial cycle $\partial\Delta_{w_1}*\dots*\partial\Delta_{w_q}$ is a boundary in 
$\sk_{\Delta(\partial\Delta_{w_1},\dots, \partial\Delta_{w_q})\cup\{i_1\}}$. This can only be the case when $\sk_{\Delta(\partial\Delta_{w_1},\dots, \partial\Delta_{w_q})\cup\{i_1\}}=\Delta(\partial\Delta_{w_1},\dots, \partial\Delta_{w_q})*\Delta(i_1)$, proving~(b).

Now we prove (a). By the previous cases, the existence of $w$ implies that $\sk$ contains $\Delta(\partial\Delta_{w_1},\dots, \partial\Delta_{w_q})$ and $\Delta(\partial\Delta_{w_1},\dots,\widehat{\partial\Delta_{w_j}}, \dots, \partial\Delta_{w_q}, i_1)$ for $j = 1,\dots, q$. The union of these subcomplexes is precisely $\partial\Delta(\partial\Delta_{w_1},\dots, \partial\Delta_{w_q}, i_1)$.

\smallskip

\underline{The case $p>1$.}

We induct on $p+q$. We have $w = \big[w_1,\dots, w_q, \mu_{i_1},\dots, \mu_{i_p}\big]$. 

To prove (b), suppose that $w=0$ but $\sk$ does not contain $\partial\Delta_{w_1}*\dots*\partial\Delta_{w_q}*\Delta(i_1,\dots, i_p)$. Then the cellular chain corresponding to $\partial\Delta_{w_1}*\dots*\partial\Delta_{w_q}*\partial\Delta(i_1,\dots, i_p)$ via \hr[Theorem]{hochster} gives a nontrivial homology class in $H_*(\zk)$. This class coincides with the Hurewicz image $h(w)$, by \hr[Lemma]{hurewicz_image}. Hence, the Whitehead product~$w$ is nontrivial. A contradiction.

Assertion~(a) is proved similarly to the case $p=1$.
\end{proof}

\begin{remark}
In our approach, the nontriviality of a higher Whitehead product $w$ is understood as the nontriviality of its canonical representative constructed in~\hr[\S]{prelimi}. Nevertheless, arguments similar to those given in the proof of the case $p=0$ show that the nontriviality assertion in 
\hr[Theorem]{smallest_realisation} remains valid if the nontriviality is understood in the classical sense, that is, as the absence of a trivial homotopy class in the set of all possible extensions.
\end{remark}

\begin{example}\label{exreal}
Consider the Whitehead product $w = \big[[\mu_1, \mu_2, \mu_3], \mu_4, \mu_5\big]$ in the moment-angle complex $\zk$ corresponding to a simplicial complex~$\sk$ on $5$ vertices. For the existence of $w$ it is necessary that the brackets $\big[[\mu_1,\mu_2,\mu_3],\mu_4\big]$, $\big[[\mu_1,\mu_2,\mu_3],\mu_5\big]$ and $[\mu_4,\mu_5]$ vanish. By \hr[Theorem]{smallest_realisation}~(b), this implies that $\sk$ contains subcomplexes $\partial\Delta(1,2,3)*\Delta(4)$, $\partial\Delta(1,2,3)*\Delta(5)$  and $\Delta(4,5)$. In other words, $\sk$ contains the complex $\partial\Delta\big(\partial\Delta(1,2,3),4,5\big)$ shown in~\hr[Figure]{12345}. 
Therefore, the latter is the smallest complex realising the Whitehead bracket $w = \big[[\mu_1, \mu_2, \mu_3], \mu_4, \mu_5\big]$.

The moment-angle complex $\zk$ corresponding to $\sK=\partial\Delta\big(\partial\Delta(1,2,3),4,5\big)$ is homotopy equivalent to the wedge of spheres $(S^5)^{\vee 4}\vee(S^6)^{\vee3}\vee S^7\vee S^8$, and each sphere is a Whitehead product, see~\cite[Example~5.4]{abra}. For example, $S^7$ corresponds to $w = \big[\big[[\mu_3, \mu_4, \mu_5], \mu_1\big], \mu_2\big]$, and $S^8$ corresponds to $w = \big[[\mu_1, \mu_2, \mu_3], \mu_4, \mu_5\big]$.
\end{example}

We expect that \hr[Theorem]{smallest_realisation} holds for all iterated higher Whitehead products:

\begin{problem}\label{proble}
	Is it true that for any iterated higher Whitehead product~$w$ the substitution complex~$\partial\Delta_{w}$ is the smallest complex realising~$w$?
\end{problem}

\section{Resolutions of the face coalgebra}\label{taylor_resolution}
Originally, cohomology of $\zk$ was described in~\cite{bu-pa00} as the $\Tor$-algebra of the face algebra of~$\sk$. As observed in~\cite{b-b-p04}, the Koszul complex calculating the $\Tor$-algebra can be identified with the cellular cochain complex of $\zk$ with respect to the standard cell decomposition.
On the other hand, the $\Tor$-algebra, and therefore cohomology of~$\zk$, can be calculated via the Taylor resolution of the face algebra as a module over the polynomial ring, see~\cite{wa-zh15},~\cite[\S4]{ayze16}. We dualise both approaches by identifying homology of $\zk$ with the $\Cotor$ of the face coalgebra of~$\sk$, and use both co-Koszul and co-Taylor resolutions to describe cycles corresponding to iterated higher Whitehead products.

Let $\k$ be a commutative ring with unit. The \emph{face algebra} $\k[\sk]$ of a simplicial complex $\sk$ is the quotient of the polynomial algebra $\k[v_1,\ldots,v_m]$ by the square-free
monomial ideal generated by non-simplices of~$\sk$:
\[
  \k[\sk]=\k[v_1,\ldots,v_m]\big/\bigl(v_{j_1}\cdots v_{j_k}\ |\
  \{j_1,\ldots,j_k\}\notin \sk\bigr).
\]
The grading is given by $\deg v_j=2$. Given a subset $J\subset[m]$, we denote by $v_J$ the square-free monomial $\prod_{j\in J}v_j$. Observe that 
\[
  \k[\sk]=\k[v_1,\ldots,v_m]\big/\bigl(v_J\ |\ J\in\mathop\mathrm{MF}(\sk)\bigr),
\]
where $\mathrm{MF}(\sk)$ denotes the set of missing faces (minimal non-faces) of~$\sk$. The face algebra $\Z[\sk]$ is also known as the \emph{face ring}, or the \emph{Stanley--Reisner ring} of~$\sk$.

We shall use the shorter notation $\k[m]$ for the polynomial algebra $\k[v_1,\ldots,v_m]$.
Let $M$ and $N$ be two $\k[m]$-modules. The $n$-th derived functor of ${}\cdot\otimes_{\k[m]}N$ is denoted by $\Tor_n^{\k[m]}(M,N)$ or $\Tor^{-n}_{\k[m]}(M,N)$. (The latter notation is better suited for topological application of the Eilenberg--Moore spectral sequence, where the $\Tor$ appears naturally as cohomology of certain spaces.) Namely, given a projective resolution $R^\bullet\to M$ with the resolvents indexed by nonpositive integers, we have
\[
  \Tor^{-n}_{\k[m]}(M,N)=H^{-n}(R^\bullet\otimes_{\k[m]}N).
\] 
The standard argument using bicomplexes and commutativity of the tensor product gives a natural isomorphism
\[
  \Tor^{-n}_{\k[m]}(M,N)\cong\Tor^{-n}_{\k[m]}(N,M).
\]

When $M$ and $N$ are graded $\k[m]$-modules, $\Tor^{-i}_{\k[m]}(M,N)$ inherits the intrinsic grading and we denote by $\Tor^{-i,2j}_{\k[m]}(M,N)$ the corresponding bigraded components.


\begin{theorem}[{\cite[Theorem~4.2.1]{bu-pa00}}]\label{zktor}
There is an isomorphism of $\k$-algebras
\[
  H^*(\zk;\k)\cong\Tor_{\k[v_1,\ldots,v_m]}\bigl(\k[\sk],\k\bigr)
\]
where the $\Tor$ is viewed as a single-graded algebra with respect to the total degree.
\end{theorem}

The $\Tor$-algebra $\Tor_{\k[m]}(\k[\sk],\k)$ can be computed either by resolving the $\k[m]$-module $\k$ and tensoring with $\k[\sk]$, or by resolving the $\k[m]$-module $\k[\sk]$ and tensoring with~$\k$.

For the first approach, there is a standard resolution of the $\k[m]$-module $\k$, the \emph{Koszul resolution}. It is defined as the acyclic differential graded algebra
\[
  \bigl(\Lambda[u_1,\ldots,u_m]\otimes\k[v_1,\ldots,v_m],d_{\k}\bigr),\quad 
  d_{\k} =\sum_i\frac\partial{\partial u_i}\otimes v_i.
\]
Here $\Lambda[u_1,\ldots,u_m]$ denotes the exterior algebra on the generators $u_i$ of cohomological degree $1$, or bidegree $(-1,2)$. After tensoring with $\k[\sk]$ we obtain the \emph{Koszul complex}
$\bigl(\Lambda[u_1,\ldots,u_m]\otimes\k[\sk],d_{\k}\bigr)$, whose cohomology is $\Tor_{\k[m]}(\k[\sk],\k)$.

Furthermore, by~\cite[Lemma~4.2.5]{bu-pa00}, the monomials $v_i^2$ and $u_iv_i$ generate an acyclic ideal in the Koszul complex. The quotient algebra
\begin{equation}\label{61}
  R^*(\sk)=\Lambda[u_1,\ldots,u_m]\otimes\k[\sk]\big/\bigl(v_i^2=u_iv_i=0,\;1\le i\le m\bigr)
\end{equation}
has a finite $\k$-basis of monomials $u_J\otimes v_I$ with $J\subset[m]$, $I\in\sk$ and $J\cap I=\varnothing$. The algebra $R^*(\sk)$ is nothing but the cellular cochain complex of $\zk$ (see~\hr[Construction]{celld}):

\begin{theorem}[\cite{b-b-p04}]\label{zkkos}
There is an isomorphism of cochain complexes
\[
  R^*(\sk)\stackrel\cong\longrightarrow\mathcal C^*(\zk),
  \quad u_J\otimes v_I\mapsto\varkappa(J,I)^*
\]
inducing the cohomology algebra isomorphism of \textup{\hr[Theorem]{zktor}}.
\end{theorem}

\begin{remark}
The isomorphism of cochain complexes in the theorem above is by inspection. The result of~\cite{b-b-p04} is that it induces an algebra isomorphism in cohomology. Also, the Koszul complex $\bigl(\Lambda[u_1,\ldots,u_m]\otimes\k[\sk],d_{\k}\bigr)$ itself can be identified with the cellular cochains of the polyhedral product $(S^\infty,S^1)^\sk$; then taking the quotient by the acyclic ideal in~\eqref{61} corresponds to the homotopy equivalence $\zk=(D^2,S^1)^\sk\stackrel\simeq\longrightarrow(S^\infty,S^1)^\sk$. See the details in~\cite[\S4.5]{bu-pa15}.
\end{remark}

In the second approach, $\Tor_{\k[m]}(\k[\sk],\k)$ is computed by resolving the $\k[m]$-module $\k[\sk]$ and tensoring with~$\k$. The \emph{minimal resolution} has a disadvantage of not supporting a multiplicative structure. There is a nice non-minimal resolution, constructed in the 1966 PhD thesis of Diana Taylor. It has a natural multiplicative structure inducing the algebra isomorphism of~\hr[Theorem]{zktor}. This \emph{Taylor resolution} of $\k[\sk]$ is defined in terms of the missing faces of $\sk$ and is therefore convenient for calculations with higher Whitehead products. We describe the resolution and its coalgebraic version next.

\begin{construction}[Taylor resolution]\label{tares}
Given a monomial ideal $(\m_1, \ldots, \m_t)$ in the polynomial algebra $\k[m]$, we define a free resolution of the $\k[m]$-module $\k[m]/(\m_1, \ldots, \m_t)$. 

For each $s = 0, \ldots, t$, let $F_s$ be a free $\k[m]$-module of rank $\binom m s$ with basis $\{e_J\}$ indexed by subsets $J\subset \{1,\ldots, t\}$ of cardinality $s$. Define a morphism $d\colon F_s \to F_{s-1}$ by
\[
  d(e_J) = \sum\limits_{j\in J}\sign (j, J) \frac{\m_J}{\m_{J\setminus j}}e_{J\setminus j},
\]
where $\m_J = \lcm_{j\in J}(\m_j)$ and $\sign (j, J)=(-1)^{n-1}$ if $j$ is the $n$-th element in the ordered set~$J$. It can be verified that $d^2 = 0$. We therefore obtain a complex
\[
  T(\m_1, \dots, \m_t) :\; 0 \to F_t \to F_{t-1} \to \dots \to F_1 \to F_0 \to 0.
\]
By the theorem of D.~Taylor, $T(\m_1, \dots, \m_t)$ is a free resolution of the $\k[m]$-module $\k[m]/(\m_1, \dots, \m_t)$. For the convenience of the reader, we include the proof of this result in the Appendix as \hr[Theorem]{taylor}.
\end{construction}

Next we describe the dualisation of the constructions above in the coalgebraic setting. The dual of $\k[v_1,\ldots,v_m]$ is the symmetric coalgebra, which we denote by $\k\langle x_1, \dots, x_m\rangle$ or~$\k\langle m\rangle$. It has a $\k$-basis consisting of monomials $\m$, with the comultiplication defined by the formula
\begin{equation}\label{comul}
	\Delta\m = \sum\limits_{\m'\cdot\m''=\m} \m'\otimes \m''.
\end{equation}
Given a set of monomials $\m_1,\ldots,\m_t$ in the variables $x_1,\ldots,x_m$, we define a subcoalgebra $C(\m_1,\ldots,\m_t)\subset\k\langle x_1, \dots, x_m\rangle$ with a $\k$-basis of monomials $\m$ that are not divisible by any of the $\m_i$, $i=1,\ldots,t$. The \emph{face coalgebra} of a simplicial complex $\sk$ is defined as
\[
  \k\langle\sk\rangle=C\bigl(x_J\ |\ J\in\mathop\mathrm{MF}(\sk)\bigr).
\]
The coalgebra $\k\langle\sk\rangle$ has a $\k$-basis of monomials $\m$ whose support is a face of~$\sk$, with the comultiplication given by~\eqref{comul}. 

Let $\Lambda$ be a coalgebra, let $A$ be a right $\Lambda$-comodule with the structure morphism $\nabla_A\colon A\to A\otimes\Lambda$, and let $B$ be a left $\Lambda$-comodule with the structure morphism $\nabla_B\colon B\to\Lambda\otimes B$. The \emph{cotensor product} of $A$ and $B$ is defined as the $\k$-comodule
\[
  A\boxtimes_{\Lambda}\!B=\ker\bigl(\nabla_A\otimes\id_B - \id_A\otimes\nabla_B\colon A\otimes B 
  \to A\otimes\Lambda\otimes B\bigr).
\]
When $\Lambda$ is cocommutative, $A\boxtimes_{\Lambda}\!B$ is a $\Lambda$-comodule.

The $n$-th derived functor of ${}\cdot\boxtimes_{\Lambda}\!B$ is denoted by $\Cotor^n_{\Lambda}(A,B)$ or $\Cotor_{-n}^{\Lambda}(A,B)$. Namely, given an injective resolution $A \to I^{\bullet}$ with the resolvents indexed by nonnegative integers, we have
\[
  \Cotor_{-n}^{\Lambda}(A,B)=\Cotor^n_{\Lambda}(A,B)=H^{n}(I^\bullet\boxtimes_{\Lambda}\!B).
\] 
If $B\to J^\bullet$ is an injective resolution of $B$, then the standard argument using a bicomplex gives isomorphisms
\begin{equation}\label{IJres}
	\Cotor^{n}_{\Lambda}(A, B) =H^{n}(I^{\bullet}\boxtimes_{\Lambda}\! B) 
  \cong H^{n}(I^{\bullet}\boxtimes_{\Lambda}\! J^{\bullet}) \cong 
  H^{n}(A\boxtimes_{\Lambda}\! J^{\bullet}).
\end{equation}

The isomorphism 
$H^{n}(I^{\bullet}\boxtimes_{\Lambda}\! B) \stackrel\cong\longrightarrow
H^{n}(A\boxtimes_{\Lambda}\! J^{\bullet})$ can be described explicitly as follows.

\begin{construction}\label{koszul_to_taylor}
Let $\eta \in H^{n}(I^{\bullet}\boxtimes_{\Lambda}\! B)$ be a homology class represented by a cycle $\eta^{(0)}\in I^n\boxtimes_{\Lambda}\! B$. We describe how to construct a cycle 
$\eta^{(n+1)}\in A\boxtimes_{\Lambda}\!J^n$ representing the same homology class in 
$\Cotor^{n}_{\Lambda}(A, B)$. 	Consider the bicomplex
\begin{center}
		\begin{tikzcd}[column sep=3.3em]
		A\boxtimes_{\Lambda}\! B \arrow{rrr}\arrow{ddd} &&& I^0\boxtimes_{\Lambda}\! B \arrow{r}\arrow{ddd} & \dots \arrow{rrr} &&& I^n\boxtimes_{\Lambda}\!B \arrow{ddd}[labl]{{ \eta^{(0)}\mapsto \partial_B(\eta^{(0)})}}
		\\\\\\
		A\boxtimes_{\Lambda}\!J^0 \arrow{rrr}\arrow{d} &&& I^0\boxtimes_{\Lambda}\! J^0 \arrow{r}\arrow{d} & \dots \arrow{rrr}{\eta^{(1)}\mapsto \partial_A(\eta^{(1)})=\partial_B(\eta^{(0)})} &&& I^n\boxtimes_{\Lambda}\! J^0 \arrow{d}
		\\
		\vdots \arrow{ddd} &&& \vdots \arrow{ddd}[labl]{{ \eta^{(n)}\mapsto \partial_B(\eta^{(n)})}}& \ddots &&& \vdots
		\\\\\\
		A\boxtimes_{\Lambda}\!J^n \arrow{rrr}{\eta^{(n+1)}\mapsto \partial_A(\eta^{(n+1)})=\partial_B(\eta^{(n)})}&&& I^0\boxtimes_{\Lambda}\! J^n \arrow{r} & \dots 
		\end{tikzcd}
\end{center}
The rows and columns are exact by the injectivity of the comodules $I^m$ and $J^l$. We have $\partial_A\big(\partial_B\eta^{(0)}\big)=-\partial_B\big(\partial_A\eta^{(0)}\big)=0$. Hence, there exists $\eta^{(1)}\in I^{n-1}\boxtimes_{\Lambda}\! J^0$ such that $\partial_A\eta^{(1)} = \partial_B\eta^{(0)}$. Similarly, there exists $\eta^{(2)}\in I^{n-2}\boxtimes_{\Lambda}\! J^1$ such that $\partial_A\eta^{(2)} = \partial_B\eta^{(1)}$. Proceeding in this fashion, we arrive at an element $\eta^{(n+1)}\in A\boxtimes_{\Lambda}\!J^n$, which represents $\eta$ by construction.
\end{construction}

We apply this construction in the following setting.  Here is the dual version of \hr[Theorem]{zktor}:

\begin{theorem}\label{zkcotor}
There is an isomorphism of $\k$-coalgebras
\[
  H_*(\zk;\k)\cong\Cotor^{\k\langle x_1,\ldots,x_m\rangle}\bigl(\k\langle\sk\rangle,\k\bigr).
\]
\end{theorem}

The coalgebra $\Cotor^{\k\langle m\rangle}(\k\langle\sk\rangle,\k)$ can be computed using the dual version of the Koszul resolution.

\begin{construction}[Koszul complex of the face coalgebra]
The \emph{Koszul resolution} for the $\k\langle m\rangle$-comodule $\k$ is defined as the acyclic differential graded coalgebra
\[
  \bigl(\k\langle x_1,\ldots,x_m\rangle\otimes\Lambda\langle y_1,\ldots,y_m\rangle,\partial_{\k}\bigr),
  \quad \partial_{\k}=\sum_i\frac\partial{\partial x_i}\otimes y_i.
\]
After cotensoring with $\k\langle\sk\rangle$ we obtain the \emph{Koszul complex}
$\bigl(\k\langle\sk\rangle\otimes\Lambda\langle y_1,\ldots,y_m\rangle,\partial_{\k}\bigr)$, whose homology is $\Cotor^{\k\langle m\rangle}(\k\langle\sk\rangle,\k)$.
\end{construction}

The relationship between the cellular chain complex of $\zk$ and the Koszul complex of $\k\langle\sk\rangle$ is described by the following dualisation of \hr[Theorem]{zkkos}.

\begin{theorem}\label{coKos}
There is an inclusion of chain complexes
\[
  \mathcal C_*(\zk)\to \bigl(\k\langle\sk\rangle\otimes
  \Lambda\langle y_1,\ldots,y_m\rangle,\partial_{\k}\bigr),\quad 
  \varkappa(J,I)\mapsto x_I\otimes y_J
\]
inducing an isomorphism in homology:
\[
  H_*(\zk;\k)\cong H\bigl(\k\langle\sk\rangle\otimes
  \Lambda\langle y_1,\ldots,y_m\rangle,\partial_{\k}\bigr)=\Cotor^{\k\langle x_1,\ldots,x_m\rangle}\bigl(\k\langle\sk\rangle,\k\bigr).
\]
\end{theorem}

On the other hand, $\Cotor^{\k\langle m\rangle}(\k\langle\sk\rangle,\k)$ can be computed using the dual version of the Taylor resolution for the $\k\langle m\rangle$-comodule $\k\langle\sk\rangle$.

\begin{construction}[Taylor resolution for comodules]\label{tacor}
Given a set of monomials $\m_1, \ldots, \m_t$, we describe a cofree resolution of the $\k\langle m\rangle$-comodule $C(\m_1, \ldots, \m_t)$. 

For each $s = 0, \ldots, t$, let $I^s$ be a cofree $\k\langle m\rangle$-comodule of rank $\binom m s$ with basis $\{e^J\}$ indexed by subsets $J\subset \{1,\ldots, t\}$ of cardinality~$s$. The differential $\partial\colon I^s \to I^{s+1}$ is defined
by
\[
  \partial(x_1^{\alpha_1}\cdots x_m^{\alpha_m}e^J)
  =\sum\limits_{j\notin J}\sign(j,J)
  \frac{x_1^{\alpha_1}\cdots x_{\vphantom{1} m}^{\alpha_m}\m_J}{\m_{J\cup\{j\}}}e^{J\cup\{j\}}.
\]
Here we assume that $\frac{x_1^{\alpha_1}\cdots x_{\vphantom{1}m}^{\alpha_{\vphantom{1}m}}\m_J}{\m_{J\cup\{j\}}}$ is zero if it is not a monomial. The resulting complex
\[
  T'(\m_1, \ldots, \m_t) :\;0\to I^0\to I^1 \to \cdots \to I^t\to 0
\]
is called the \emph{Taylor resolution} of the $\k\langle m\rangle$-comodule $C(\m_1, \ldots, \m_t)$. The proof that it is indeed a resolution is given in~\hr[Theorem]{taylor}.
\end{construction}

\begin{construction}[Taylor complex of the face coalgebra]
Let  $\k\langle\sk\rangle=C\bigl(x_J\ |\ J\in\mathop\mathrm{MF}(\sk)\bigr)$ be the face coalgebra of a simplicial complex~$\sk$. In this case it is convenient to view the $s$-th term $I^s$ in the Taylor resolution as the cofree $\k\langle m\rangle$-comodule with basis consisting of exterior monomials
$w_{J_1}\wedge\cdots\wedge w_{J_s}$, where $J_1,\ldots,J_s$ are different missing faces of~$\sk$. The differential then takes the form
\[
  \partial_{\k\langle\sk\rangle}(x_1^{\alpha_1}\cdots x_{\vphantom{1} m}^{\alpha_m}\cdot 
  w_{J_1}\wedge\dots\wedge   w_{J_s})= \sum\limits_{J \neq J_1, \ldots, J_s}   
  \frac{x_1^{\alpha_1}\cdots x_{\vphantom{1} m}^{\alpha_m}}
  {x_{(J_1\cup\dots\cup J_s\cup J)\setminus (J_1\cup\dots\cup J_s)}}
  \cdot w_J\wedge w_{J_1}\wedge\dots\wedge w_{J_s}
\]
(the sum is taken over missing faces $J\in\mathop\mathrm{MF}(\sK)$ different from $J_1,\ldots,J_s$). 

After cotensoring with $\k$ over $\k\langle m\rangle$ we obtain the \emph{Taylor complex} of $\k\langle\sk\rangle$ calculating $\Cotor^{\k\langle x_1,\ldots,x_m\rangle}\bigl(\k\langle\sk\rangle,\k\bigr)$. Its $(-s)$th graded component is a free $\k$-module with basis of exterior monomials
$w_{J_1}\wedge\cdots\wedge w_{J_s}$, where $J_1,\ldots,J_s$ are different missing faces of~$\sk$. The differential is given by
\[
  \partial_{\k\langle\sk\rangle}(w_{J_1}\wedge\dots\wedge   w_{J_s})= \sum\limits_{J\subset J_1\cup\dots\cup J_s}
  w_J\wedge w_{J_1}\wedge\dots\wedge w_{J_s}
\]
(the sum is over missing faces $J\subset J_1\cup\dots\cup J_s$ different from any of the $J_1,\ldots,J_s$).
\end{construction}

We therefore have two methods of calculating $H_*(\zk)=\Cotor^{\k\langle x_1,\ldots,x_m\rangle}\bigl(\k\langle\sk\rangle,\k\bigr)$: by resolving $\k$ (Koszul resolution) or by resolving $\k\langle\sK\rangle$ (Taylor resolution). The two resulting complexes are related by the chain of quasi-isomorphisms~\eqref{IJres} and \hr[Construction]{koszul_to_taylor}.

\begin{example} Let $\sk$ be the substitution complex $\partial\Delta\big(\partial\Delta(1,2,3), 4,5\big)$, see \hr[Figure]{12345}.	
After tensoring the Taylor resolution for $\Z\langle\sk\rangle$ with $\Z$ we obtain the following complex:
	\begin{center}
		\begin{tikzpicture}[every node/.style={midway}]
		\matrix[column sep=1.3cm,row sep={4cm,between origins}] at (0, 0) {
			\node(A) {$\Z$};   &
			\node(B) {$\Z^4$}; &
			\node(C) {$\Z^6$}; &&&&&&&
			\node(D) {$\Z^4$}; \\
			& \node(E) {$\Z$}; &&&&&&&& \node(G){};\\};
		\draw[->] (A) -- (B) node[below] {\small $1\mapsto 0$};
		\draw[->] (B) -- (C) node[below=0.1em]
		{\small $\begin{aligned}
			&w_{123}\mapsto 0\\
			&w_{145}\mapsto 0\\
			&w_{245}\mapsto 0\\
			&w_{345}\mapsto 0\\
			\end{aligned}$};
		
		\draw[->] (C) -- (D) node[below=0.1em]
		{\small $\begin{aligned}
			&w_{123}\wedge w_{145}\mapsto \phantom{-} w_{123}\wedge w_{145}\wedge w_{245} + w_{123}\wedge w_{145}\wedge w_{345}\\
			&w_{123}\wedge w_{245}\mapsto  - w_{123}\wedge w_{145}\wedge w_{245}+  
       w_{123}\wedge w_{245}\wedge w_{345} \\
			&w_{123}\wedge w_{345}\mapsto 
			-w_{123}\wedge w_{145}\wedge w_{345} - w_{123}\wedge w_{245}\wedge w_{345}\\
			&w_{145}\wedge w_{245}\mapsto 0\\
			&w_{145}\wedge w_{345}\mapsto 0\\
			&w_{245}\wedge w_{345}\mapsto 0
			\end{aligned}
			$};
		
		\draw[->,rounded corners] (D.south) -- +(0,-1) |- 
		node[below=0.1em, pos=.8]
		{\small $\begin{aligned}
			-w_{123}\wedge w_{145}\wedge w_{245}\wedge w_{345}\mapsfrom w_{123}\wedge w_{145}\wedge w_{245}&\\
			\phantom{-}w_{123}\wedge w_{145}\wedge w_{245}\wedge w_{345}\mapsfrom w_{123}\wedge w_{145}\wedge w_{345}&\\
			-w_{123}\mapsfrom w_{145}\wedge w_{245}\wedge w_{345}\mapsfrom w_{123}\wedge w_{245}\wedge w_{345}&\\
			\phantom{-}w_{123}\wedge w_{145}\wedge w_{245}\wedge w_{345}\mapsfrom w_{145}\wedge w_{245}\wedge w_{345}&
			\end{aligned}	
			$} (E);
		\end{tikzpicture}
	\end{center}
We see that homology of this complex agrees with homology of the wedge $(S^5)^{\vee 4}\vee(S^6)^{\vee3}\vee S^7\vee S^8$, in accordance with \hr[Example]{exreal}. 
\end{example}

\section{Higher Whitehead products and Taylor resolution}
Given an iterated higher Whitehead product $w$, \hr[Lemma]{hurewicz_image} gives a canonical cellular cycle representing the Hurewicz image of~$w$. By \hr[Theorem]{coKos}, this cellular cycle can be viewed as a cycle in the Koszul complex calculating $\Cotor^{\k\langle m\rangle}\bigl(\k\langle\sk\rangle,\k\bigr)$. Here we use \hr[Construction]{koszul_to_taylor} to describe a canonical cycle representing an iterated higher Whitehead product $w$ in the coalgebraic Taylor resolution. This gives a new criterion for the realisability of~$w$.

\begin{theorem}\label{tarep}
Let $w$ be a nested iterated higher Whitehead product
\begin{equation}\label{taylor_whitehead_product}
  w = \Big[\big[\dots\big[[\mu_{i_{11}},\ldots, \mu_{i_{1p_1}}], \mu_{i_{21}},\dots, \mu_{i_{2p_2}}\big],  
  \dots\big], \mu_{i_{n1}},\dots, \mu_{i_{np_n}}\Big].	
\end{equation}
	Then the Hurewicz image $h(w) \in H_*(\zk) = \mathop\mathrm{Cotor}\nolimits^{\Z\langle m\rangle}( \Z\langle\sk\rangle, \Z)$ is represented by the following cycle in the Taylor complex of ${\Z\langle\sk\rangle}$
\begin{equation}\label{taylor_cycle}
  \bigwedge\limits_{k=1}^n\hspace{3mm}\Biggl(\hspace{1mm}\sum\limits_{\substack{J\in   
  \mathop\mathrm{MF}(\sk) \\ J\setminus
  \bigl(\bigcup\limits_{j=1}^{n-k} I_j\bigr)   =I_{n-k+1} }}
\hspace{-5mm}w_J\hspace{1mm}\Biggr),
\end{equation}
where $I_k = \{i_{k1}, \dots, i_{kp_k}\}$. 
\end{theorem}

\begin{proof}
Recall from \hr[Construction]{celld} that for a given pair of non-intersecting index sets $I = \{i_1, \dots, i_s\}$ and  $J = \{j_1, \dots, j_t\}$ we have a cell
\[
  \varkappa(J,I) = D_{i_1}\cdots D_{i_s} S_{j_1}\cdots S_{j_t}.
\]
It belongs to $\zk$ whenever $I\in\sK$. Using this notation we can rewrite the canonical cellular chain~$h_c(w)$ from \hr[Lemma]{hurewicz_image_old} as follows:
\begin{equation}\label{chain_new_notation}
	h_c(w) = \prod\limits_{k = 1}^n \left(\sum\limits_{I \in \partial\Delta(I_k)} 
  \varkappa\big(I_k\setminus I, I\big) \right).
\end{equation}
Here and below the sum is over maximal simplicies $I \in \partial\Delta(I_k)$ only (otherwise the right hand side above is not a homogeneous element).

Now we apply \hr[Construction]{koszul_to_taylor} to \eqref{chain_new_notation}. We obtain the following zigzag of elements in the bicomplex relating the Koszul complex with differential $\partial_\Z$ to the Taylor complex with differential $\partial_{\Z\langle\sk\rangle}$:
\begin{center}
  \begin{tikzpicture}[xscale=1, yscale=1]
  \node(A) at (0,4)
  {$\varkappa(\varnothing, I_1)\prod\limits_{k=2}^{n} \Bigl(\sum\limits_{I \in \partial\Delta(I_k)} 
  \varkappa\big(I_k\setminus I, I\big) \Bigr)$}; 
  \node(B) at (7.8,4)
  {$\prod\limits_{k = 1}^n \Bigl(\sum\limits_{I \in \partial\Delta(I_k)} \varkappa
  \big(I_k\setminus I, I\big) \Bigr)$};
  \node(C) at (7.8,2)
  {$\prod\limits_{k=2}^{n}
  \Bigl(\sum\limits_{I \in \partial\Delta(I_k)} \varkappa\big(I_k\setminus I, I\big) \Bigr)w_{I_1}$};
  \node(D) at (0,2) 
  {$\varkappa(\varnothing, I_{2})\prod\limits_{k=3}^{n}
  \Bigl(\sum\limits_{I \in \partial\Delta(I_k)} \varkappa\big(I_k\setminus I, I\big) \Bigr)w_{I_1}$};
  \node(E) at (7.8,0) {$\prod\limits_{k=3}^{n}
  \Bigl(\sum\limits_{I \in \partial\Delta(I_k)} \varkappa\big(I_k\setminus I, I\big) \Bigr)
  \Bigl(\sum\limits_{(J\setminus I_1)=I_2}w_J\Bigr)\wedge w_{I_1}$};
  \node(F) at (0,0){\hphantom{aaaaaaaaaa}$\cdots$\hphantom{aaaaaaaaaa}};
  \path[|->, font=\scriptsize]
  (A) edge node[above]{$\partial_{\Z}$} (B)
	(A) edge node[above]{$\partial_{\Z\langle\sk\rangle}$} (C)
	(D) edge node[above]{$\partial_{\Z}$} (C)
	(D) edge node[above]{$\partial_{\Z\langle\sk\rangle}$} (E)
   (F) edge node[above] {$\partial_{\Z}$} (E);
	\end{tikzpicture}
\end{center}
It ends up precisely at element~\eqref{taylor_cycle} in the Taylor complex.
\end{proof}

\begin{example}	
Once again consider the complex $\sk = \partial\Delta(\partial\Delta(1,2,3),4,5)$ shown in \hr[Figure]{12345}. We have $\zk\simeq(S^5)^{\vee 4}\vee(S^6)^{\vee3}\vee S^7\vee S^8$ by~\cite[Example~5.4]{abra}, and each sphere is a Whitehead product. These Whitehead products together with the representing cycles in the Koszul and Taylor complexes are shown in \hr[Table]{t1} for each sphere.
\begin{table}[!h]
	\renewcommand{\arraystretch}{1.7}
	\begin{center}
		\begin{tabular}{|c|c|c|}
       \hline
			\footnotesize{Whitehead product}
			&
			\footnotesize{Koszul (cellular) cycle}
			&
			\footnotesize{Taylor cycle}
       \\
			\hline
 			{\footnotesize $[\mu_1, \mu_2, \mu_3]$}
        &
			{\footnotesize $D_1D_2S_3 + D_1S_2D_3 + S_1D_2D_3$}
			&
			{\footnotesize $w_{123}$}
			\\
			\hline
			{\footnotesize $[\mu_1, \mu_4, \mu_5]$}
			&
       {\footnotesize $D_1D_4S_5 + D_1S_4D_5 + S_1D_4D_5$}
			&
			{\footnotesize $w_{145}$}
			\\
			\hline
			{\footnotesize $[\mu_2, \mu_4, \mu_5]$}
        &
			{\footnotesize $D_2D_4S_5 + D_2S_4D_5 + S_2D_4D_5$}
			&
			{\footnotesize $w_{245}$}
			\\
			\hline
			{\footnotesize $[\mu_3, \mu_4, \mu_5]$}
        &
			{\footnotesize $D_3D_4S_5 + D_3S_4D_5 + S_3D_4D_5$}
			&
			{\footnotesize $w_{345}$}
			\\
			\hline
       {\footnotesize $\big[[\mu_1, \mu_4, \mu_5],\mu_2\big]$}
			&
			{\footnotesize $(D_1D_4S_5 + D_1S_4D_5 + S_1D_4D_5)S_2$}
			&
			{\footnotesize $w_{245}\wedge w_{145}$}
			\\
			\hline
        {\footnotesize $\big[[\mu_1, \mu_4, \mu_5],\mu_3\big]$}
			&
			{\footnotesize $(D_1D_4S_5 + D_1S_4D_5 + S_1D_4D_5)S_3$}
			&
			{\footnotesize $w_{345}\wedge w_{145}$}
			\\
			\hline
        {\footnotesize $\big[[\mu_2, \mu_4, \mu_5],\mu_3\big]$}
			&
			{\footnotesize $(D_2D_4S_5 + D_2S_4D_5 + S_2D_4D_5)S_3$}
			&
			{\footnotesize $w_{345}\wedge w_{245}$}
			\\
			\hline
       {\footnotesize $\big[\big[[\mu_1,\mu_4,\mu_5],\mu_2\big]\mu_3\big]$}
			&
			{\footnotesize $(D_1D_4S_5 + D_1S_4D_5 + S_1D_4D_5)S_2S_3$}
			&
			{\footnotesize $(w_{123} + w_{345})\wedge w_{245}\wedge w_{145}$}
			\\
			\hline
       {\footnotesize $\big[[\mu_1,\mu_2,\mu_3],\mu_4,\mu_5\big]$}
			&
			{\footnotesize $(D_1D_2S_3 + D_1S_2D_3 + S_1D_2D_3)(D_4S_5+S_4D_5)$}
			&
			{\footnotesize $(w_{145}+w_{245}+w_{345}) \wedge w_{123}$}
			\\
			\hline
		\end{tabular}
		\caption{Koszul and Taylor cycles representing Whitehead products}\label{t1}
	\end{center}
\end{table}
\end{example}

An important feature of the Taylor cycle~\eqref{taylor_cycle} is that it has the form of a product of sums of generators $w_J$ corresponding to missing faces, and the rightmost factor is a \emph{single} generator~$w_{I_1}$. This can be seen in the right column of \hr[Table]{t1}. Below we give an example of a Taylor cycle which \emph{does not} have this form. It corresponds to a sphere which \emph{is not} a Whitehead product, although the corresponding $\zk$ is a wedge of spheres. This example was discovered in~\cite[\S7]{abra}. 

\begin{example}
Consider the simplicial complex
\begin{multline*}
	\sk=\partial\Delta\bigl(\partial\Delta(1,2,3),4,5,6\bigr)\cup\Delta(1,2,3)  
  \\=\bigl(\partial\Delta(1,2,3)*\partial\Delta(4,5,6)\bigr)\cup\Delta(1,2,3)\cup\Delta(4,5,6). 
\end{multline*}
We have $\zk\simeq(S^7)^{\vee 6}\vee (S^8)^{\vee 6}\vee (S^9)^{\vee 2}\vee S^{10}$, 
see~\cite[Proposition~7.1]{abra}. Here is the staircase diagram of \hr[Construction]{koszul_to_taylor} relating the Koszul and Taylor cycles corresponding to~$S^{10}$: 
\begin{center}
	\begin{tikzpicture}[xscale=1, yscale=1]
	\node(A) at (0,4.5)
	{$D_1D_2D_3(D_4D_5S_6 + D_4S_5D_6 + S_4D_5D_6)$}; 
	\node(B) at (5,6)
	{$(D_1D_2S_3 + D_1S_2D_3 + S_1D_2D_3)(D_4D_5S_6 + D_4S_5D_6 + S_4D_5D_6)$};
	\node(C) at (5,3)
	{$(D_5S_6 + S_5D_6)w_{1234}+(D_4S_6 + S_4D_6)w_{1235}+(D_4S_5 + S_4D_5)w_{1236}$};
	\node(D) at (0,1.5) 
	{$D_5D_6w_{1234}+D_4D_6w_{1235}+D_4D_5w_{1236}$};
	\node(E) at (5,0) {$-(w_{1234}+w_{1235}+w_{1236})\wedge(w_{1456}+w_{2456}+w_{3456})$};
	\path[|->, font=\scriptsize]
	(A) edge node[above, near start]{$\partial_{\Z}$} (B)
	(A) edge node[above, near end]{$\partial_{\Z\langle\sk\rangle}$} (C)
	(D) edge node[above, near start]{$\partial_{\Z}$} (C)
	(D) edge node[above, near end]{$\partial_{\Z\langle\sk\rangle}$} (E);
	\end{tikzpicture}
\end{center}
We see that the Taylor cycle does not have a factor consisting of a single generator~$w_J$. This reflects the fact that the sphere $S^{10}$ in the wedge is not an iterated higher Whitehead product, see~\cite[Proposition~7.2]{abra}.
\end{example}

Using the same argument as in the proof of \hr[Theorem]{tarep}, we can write down the Taylor cycle representing the Hurewicz image of an \emph{arbitrary} iterated higher Whitehead product, not only a nested one. The general form of the answer is rather cumbersome though. Instead of writing a general formula, we illustrate it on an example.

\begin{example} Consider the substitution complex $\sk = \partial\Delta\big(\partial\Delta(1,2,3), \partial\Delta(4,5,6),7,8\big)$. By~\hr[Theorem]{realisation}, it realises the Whitehead product $w=\big[[\mu_1,\mu_2,\mu_3], [\mu_4,\mu_5,\mu_6],\mu_7,\mu_8\big]$. From the description of the missing faces in \hr[Definition]{substitution_definition} we obtain
\begin{align*}
  \mathop\mathrm{MF}(\sk) = \bigl\{&\Delta(1,2,3), \Delta(4,5,6),\Delta(1,4,7,8),\Delta(1,5,7,8),
  \Delta(1,6,7,8),\\
  &\Delta(2,4,7,8),\Delta(2,5,7,8),\Delta(2,6,7,8),\Delta(3,4,7,8),\Delta(3,5,7,8),\Delta(3,6,7,8)\bigr\}.
\end{align*}
Applying \hr[Construction]{koszul_to_taylor} to the canonical cellular cycle
\[
  h_c(w) = (D_1D_2S_3 + D_1S_2D_3 + S_1D_2D_3)(D_4D_5S_6 + D_4S_5D_6 + S_4D_5D_6)  
  (D_7S_8+S_7D_8)
\]
we obtain the corresponding cycle in the Taylor complex:
\[
  (w_{1478}+w_{1578}+w_{1678}+w_{2478}+w_{2578}+w_{2678}
  +w_{3478}+w_{3578}+w_{3678})\wedge w_{456}\wedge w_{123}.
\]
\end{example}

\appendix\section{Proof of Taylor's theorem}
Here we prove that the complex $T(\m_1, \dots, \m_t)$ introduced in \hr[Construction]{tares} is a free resolution and the complex $T'(\m_1, \dots, \m_t)$ from \hr[Construction]{tacor} is a cofree resolution. In the case of modules, the argument was outlined in~\cite[Exercise~17.11]{eise95} (see also~\cite[Theorem~7.1.1]{he-hi11}). The comodule case is obtained by dualisation.

\begin{theorem}\label{taylor}\ 
\begin{itemize}[leftmargin=0.06\textwidth]
\item[\textup(a\textup)]
$T(\m_1, \dots, \m_t)$ is a free resolution of the $\k[m]$-module 
$\k[m]/(\m_1, \dots, \m_t)$.

\item[\textup(b\textup)]
$T'(\m_1, \dots, \m_t)$ is a cofree resolution of the $\k\langle m\rangle$-comodule $C(\m_1, \dots, \m_t)$.
\end{itemize}
\end{theorem}
\begin{proof}
Denote $\n_i = \frac{\m_i}{\gcd(\m_i, \m_t)}$. 
Then we have\footnote{Given ideals $\mathcal {I,J}$ in a commutative ring $R$, the ideal quotient is defined as $(\mathcal I : \mathcal J) = \{f\in R\,|\, f\mathcal J \subset \mathcal I\}$.}	$(\m_1,\dots, \m_{t-1} : \m_t) = (\n_1, \dots, \n_{t-1})$.

In the case of modules, there is a short exact sequence
\[
  0\to \k[m]/(\n_1, \dots, \n_{t-1}) \xrightarrow{\cdot \m_t} \k[m]/(\m_1, \dots, \m_{t-1})\to 
  \k[m]/(\m_1, \dots, \m_t) \to 0.
\]
Assume by induction that $T(\m_1, \dots, \m_{t-1})$ is a resolution. Consider the injective morphism
\[  
  \varphi\colon \k[m]/(\n_1, \dots, \n_{t-1}) 
  \stackrel{\cdot \m_t}\longrightarrow \k[m]/(\m_1, \dots, \m_{t-1})
\]
and the induced morphism of resolutions
\[
  \widetilde\vphi\colon T(\n_1, \dots, \n_{t-1}) \to T(\m_1, \dots, \m_{t-1}).
\]
The proof consists of three lemmata, proved separately below. By \hr[Lemma]{conta}, the complex $T(\m_1, \dots, \m_t)$ can be identified with the cone of the morphism $\widetilde\vphi$. Then \hyperref[cone_of_morphism]{Lemma~\ref{cone_of_morphism}} implies that $T(\m_1, \dots, \m_t)$ is a resolution for $\k[m]/(\m_1, \dots, \m_t)$.

Similarly, in the comodule case we consider the short exact sequence of comodules
\[
  0\to C(\m_1, \dots, \m_t) \to C(\m_1, \dots, \m_{t-1})\xrightarrow{\cdot \frac{1}{\m_t}} 
  C(\n_1, \dots, \n_{t-1})\to 0,
\]
use induction, and apply the lemmata below.
\end{proof}

\begin{lemma}\label{cone_of_morphism}\ 
\begin{itemize}[leftmargin=0.06\textwidth]
	\item[\textup(a\textup)]
	Let $\vphi\colon\bar V \to V$ be an injective morphism of modules. Let $\bar U_{\bullet} \to\bar V$ and $U_{\bullet} \to V$ be resolutions. Then the cone $C(\widetilde \vphi)$ of the induced morphism of resolutions $\widetilde\vphi\colon\bar U_{\bullet} \to U_{\bullet}$ is a resolution for $V/\vphi(\bar V)$.

\item[\textup(b\textup)]
Let $\vphi'\colon A \to\bar A$ be a surjective morphism of comodules. Let $A \to B^{\bullet}$ and $\bar A \to\bar B^{\bullet}$ be resolutions. Then the cocone $C'(\widetilde \vphi')$ of the induced morphism of resolutions $\widetilde\vphi'\colon B^{\bullet} \to\bar B^{\bullet}$ is a resolution for $\ker(\varphi'\colon A \to\bar A)$.
\end{itemize}
\end{lemma}
\begin{proof}
Consider the homology long exact sequence associated with the cone $C(\widetilde \vphi)$:
\[
	\begin{array}{ccccccccccc}
	\cdots\longrightarrow& H_1(U_\bullet) & \longrightarrow
	& H_1(C(\widetilde\varphi)) & \longrightarrow& H_0(\bar U_\bullet) & \longrightarrow
	& H_0(U_\bullet) & \longrightarrow & H_0(C(\widetilde\varphi)) &\to 0\\[3pt]
	&\|&&&&\|&&\|&&\|\\[2pt]
	&0&&&&\bar V&\stackrel\varphi\longrightarrow&V&\longrightarrow&V/\varphi(\bar V)&\to0
	\end{array}
\]
Injectivity of $\varphi\colon\bar V\to V$ implies that $H_1(C(\widetilde{\varphi}))=0$. Vanishing of the higher homology groups $H_i\big(C(\widetilde{\varphi})\big)$, $i>1$, follows from the exactness. Hence, $C(\widetilde{\varphi})$ is a resolution for $H_0(C(\widetilde{\varphi})) \cong V/\varphi(\bar V)$.

The comodule case is proved by straightforward dualisation.
%
\end{proof}

\begin{lemma}\
\begin{itemize}[leftmargin=0.06\textwidth]
	\item[\textup(a\textup)]The morphism $\widetilde \varphi\colon T(\n_1,\dots, \n_{t-1}) \to T(\m_1, \dots, \m_{t-1})$ is given by
\[
  \widetilde\vphi(\bar e_J) = \frac{\m_{J\cup \{t\}}}{\m_J}e_J,\qquad J\subset \{1, \dots, t-1\}.
\]

	\item[\textup(b\textup)]
The morphism $\widetilde \varphi'\colon T'(\m_1, \dots, \m_{t-1}) \to T'(\n_1,\dots, \n_{t-1})$ is given by
\[
  \widetilde\varphi'(x_1^{\alpha_1}\cdots x_{\vphantom{1} m}^{\alpha_m}e^J) = 
  \frac{\m_J}{\m_{J\cup \{t\}}}
  x_1^{\alpha_1}\cdots x_{\vphantom{1} m}^{\alpha_m}\bar e^J,\qquad J\subset \{1, \dots, t-1\}.
\]
\end{itemize}
\end{lemma}
\begin{proof}
We need to show that the described maps commute with the differentials, as this property defines a morphism of resolutions uniquely. 

For (a), denote $T(\n_1,\dots, \n_{t-1}) = \{\bar F_{\bullet}, \bar d\}$ and $T(\m_1,\dots, \m_{t-1}) = \{F_{\bullet}, d\}$. Recall that $F_{\bullet}$ has basis $\{e^J\}$ indexed by subsets $J\subset \{1,\ldots, t-1\}$, and denote the corresponding basis elements of $\bar F_{\bullet}$ by $\bar e^J$. The required property follows by considering the diagram
\begin{center}
\begin{tikzcd}[column sep=2em,row sep=2em]
  \bar F_s \ar[rrr, "\bar d"] \ar[dddd, "\widetilde\vphi"] &&& \bar F_{s-1} \ar[dddd, "\widetilde\vphi"] \\
	& \bar e_J \ar[r, "\bar d", maps to] \ar[dd, "\widetilde\vphi", maps to] & 
  \sum\limits_{j\in J}\sign(j, J)\frac{\n_J}{\n_{J\setminus \{j\}}}\bar e_{J\setminus\{j\}} 
  \ar[d, "\widetilde\vphi", maps to] \\ 
  && \sum\limits_{j\in J}\sign(j, J)\frac{\n_J}{\n_{J\setminus \{j\}}}
  \frac{\m_{(J\setminus \{j\})\cup\{t\}}}{\m_{J\setminus \{j\}}}e_{J\setminus\{j\}} \\ 
  & \frac{\m_{J\cup\{t\}}}{\m_J}e_J \ar[r, "d", maps to] & \sum\limits_{j\in J}\sign(j, J)
  \frac{\m_J}{\m_{J\setminus \{j\}}}\frac{\m_{J\cup\{t\}}}{\m_J}e_{J\setminus\{j\}} \ar[u, equal] \\
  F_s  \ar[rrr, "d"] &&& F_{s-1}.
\end{tikzcd}
\end{center}
Here we used the identity
\[
  \frac{\m_{J\cup\{t\}}}{\m_{(J\setminus\{j\})\cup\{t\}}}=\frac{\n_{J}}{\n_{J\setminus\{j\}}},
\]
which follows from the defintion of $\n_i$.

Statement (b) is proved by dualisation.
\end{proof}

\begin{lemma}\label{conta}
Up to a sign in the differentials, 
\begin{itemize}[leftmargin=0.06\textwidth]
	\item[\textup(a\textup)]
	the cone complex $C(\widetilde\vphi)$ is isomorphic to $T(\m_1, \dots, \m_t)$;

\item[\textup(b\textup)]
the cocone complex $C'(\widetilde\vphi')$ is isomorphic to $T'(\m_1, \dots, \m_t)$.
\end{itemize}
\end{lemma}
\begin{proof}
For (a), we denote $T(\n_1,\dots, \n_{t-1}) = \{\bar F_{\bullet}, \bar d\}$, $T(\m_1,\dots, \m_{t-1}) = \{F_{\bullet}, d\}$ and $T(\m_1, \dots, \m_t) = \{\widetilde F_{\bullet}, \widetilde d\}$.
	
We shall define a morphism $\psi\colon C(\widetilde{\varphi})\to T(\m_1, \dots, \m_t)$, that is, $\psi\colon \bar F_s\oplus F_{s+1}\to\widetilde F_{s+1}$ commuting with the differentials. As $F_\bullet$ is a subcomplex of both $C(\widetilde{\varphi})$ and $\widetilde F_\bullet$, we define $\psi$ on $e_J\in F_{s+1}$ by $\psi(e_J) = \widetilde e_J$. Now we define $\psi$ on $\bar e_J\in\bar F_s$ by the formula $\psi(\bar e_J) = \widetilde{e}_{J\cup\{t\}}$. The following diagram shows that the resulting map $\psi$ indeed commutes with the differentials:
\begin{center}
\begin{tikzcd}[column sep=2em,row sep=2em]
  \bar F_s\oplus F_{s+1} \ar[rrr, "d_{C(\widetilde\vphi)}"] \ar[ddd, "\psi"]
  &&&\bar F_{s-1}\oplus F_s \ar[ddd, "\psi"] \\
  &\bar e_J \ar[r, "\widetilde\vphi-\bar d", maps to] \ar[d, "\psi", maps to] & 
  \frac{\m_{J\cup\{t\}}}{\m_J}e_J - \sum\limits_{j\in J}\sign(j, J)
  \frac{\n_J}{\n_{J\setminus \{j\}}}\bar e_{J\setminus\{j\}} \ar[d, "\psi", maps to] \\ 
  & \pm \widetilde e_{J\cup\{t\}} \ar[r, "\widetilde d", maps to] & 
  \frac{\m_{J\cup\{t\}}}{\m_J}\widetilde e_J \pm \sum\limits_{j\in J}\sign(j, J)\frac{\n_J}
  {\n_{J\setminus \{j\}}}\widetilde e_{(J\setminus\{j\})\cup\{t\}} \\ 
  \widetilde F_{s+1}  \ar[rrr, "\widetilde d"] &&& \widetilde F_s.
\end{tikzcd}
\end{center}
Thus, $\psi$ defines a morphism $C(\widetilde{\varphi})\to T(\m_1, \dots, \m_t)$, which is clearly an isomorphism.

\smallskip

For (b), we use the notation $T'(\n_1,\dots, \n_{t-1}) = \{\bar I^{\bullet}, \bar\partial\}$, $T'(\m_1,\dots, \m_{t-1}) = \{I^{\bullet}, \partial\}$, and $T'(\m_1, \dots, \m_t) = \{\widetilde I^{\bullet}, \widetilde \partial\}$.
	
We define $\psi'\colon T'(\m_1, \dots, \m_t)\to C'(\widetilde{\varphi})$, that is, $\psi'\colon \widetilde I^s\to I^s\oplus \bar I^{s-1}$ by the formula
\[
  \psi'(x_1^{\alpha_1}\cdots x_{\vphantom{1} m}^{\alpha_m}\widetilde e^J) =
  \begin{cases}
  (-1^{|J|-1}x_1^{\alpha_1}\cdots x_{\vphantom{1} m}^{\alpha_m}\bar e^{J\setminus\{t\}},
  \quad&\text{for $t\in J$},\\[2pt]
  (-1)^{|J|}x_1^{\alpha_1}\cdots x_{\vphantom{1} m}^{\alpha_m} e^J, &\text{for $t\notin J$.}
	\end{cases}
\]
We need to check that $\psi'$ commutes with the differentials. For $t\in J$ we have
\begin{center}
\begin{tikzcd}[column sep=2em,row sep=2em]
  x_1^{\alpha_1}\cdots x_{\vphantom{1} m}^{\alpha_m}\widetilde e^J
  \ar[r,"\widetilde\partial",maps to]\ar[d, "\psi' ", maps to] &
  \sum\limits_{j\notin J}\sign(j,J)\frac{x_1^{\alpha_1}\cdots x_{\vphantom{1}m}^{\alpha_m}\m_J}
  {\m_{J\cup\{j\}}}\widetilde e^{J\cup\{j\}} \ar[d,"-\psi' ",maps to] \\
  (-1)^{|J|-1}x_1^{\alpha_1}\cdots x_{\vphantom{1} m}^{\alpha_m}\bar e^{J\setminus\{t\} }
  \ar[r,"\bar\partial",maps to] &
	(-1)^{|J|-1}\sum\limits_{j\notin J}\sign(j,J)
  \frac{x_1^{\alpha_1}\cdots x_{\vphantom{1} m}^{\alpha_m}\n_J}
  {\n_{J\cup\{j\}}}\bar e^{J\cup\{j\}\setminus\{t\} }.
\end{tikzcd}
\end{center}
For $t\notin J$ we have
\begin{center}
\begin{tikzcd}[column sep=2em,row sep=2em]
  x_1^{\alpha_1}\cdots x_{\vphantom{1} m}^{\alpha_m}\widetilde e^J
  \ar[r,"\widetilde\partial",maps to]\ar[d, "\psi' ", maps to] &
  \sum\limits_{j\notin J,\; j \neq t}\sign(j,J)
  \frac{x_1^{\alpha_1}\cdots x_{\vphantom{1} m}^{\alpha_m}\m_J}
  {\m_{J\cup\{j\}}}\widetilde e^{J\cup\{j\}} 
  + (-1)^{|J|}\frac{x_1^{\alpha_1}\cdots x_{\vphantom{1} m}^{\alpha_m}\m_J}    
  {\m_{J\cup\{t\}}}\widetilde e^{J\cup\{t\}} \ar[d,"\psi' ",maps to]\\
  x_1^{\alpha_1}\cdots x_{\vphantom{1} m}^{\alpha_m}e^J
  \ar[r," - \partial + \widetilde\varphi' ",maps to] &
  -\sum\limits_{j\notin J,\; j \neq t}\sign(j,J)
  \frac{x_1^{\alpha_1}\cdots x_{\vphantom{1} m}^{\alpha_m}\m_J}
  {\m_{J\cup\{j\}}}e^{J\cup\{j\}} + 
  \frac{x_1^{\alpha_1}\cdots x_{\vphantom{1} m}^{\alpha_m}\m_J}
  {\m_{J\cup\{t\}}}\bar e^J;
\end{tikzcd}
\end{center}
%
We therefore obtain the required isomorphism $\psi'\colon T'(\m_1, \dots, \m_t) \to C'(\widetilde{\varphi})$.
\end{proof}

\end{document}